\documentclass[11pt]{amsart}
\usepackage{amsmath, amsthm}
\usepackage{euscript}
\usepackage{amssymb,latexsym,amsfonts,amstext}
\usepackage{amsthm}
\usepackage{amsopn}
\usepackage{graphicx,subfigure}
\usepackage[scanall]{psfrag}
\usepackage{setspace}
\newtheorem{thm}{Theorem}[section]
\newtheorem*{definition}{Definition}
\newtheorem{lemma}[thm]{Lemma}
\newtheorem{prop}[thm]{Proposition}
\newtheorem{coro}[thm]{Corollary}
\newtheorem{eg}{Example}[section]
\usepackage[margin=1in]{geometry}
\makeatletter
\renewcommand\section{\@startsection {section}{1}{\z@}{- \baselineskip}{\baselineskip}{\bfseries\scshape\Large}}
\makeatother
\makeatletter
\renewcommand\subsection{\@startsection {subsection}{2}{\z@}{-\baselineskip}{\baselineskip}{\large\bfseries}}
\makeatother
\newcommand{\ags}{A(\Gamma^\sigma)}

\begin{document}

\title{Representations of $Aut(A(\Gamma))$ acting\\
on homogeneous components of $A(\Gamma)$ and $A(\Gamma)^!$}

\author{Colleen Duffy}
\address{Rutgers, the State University of New Jersey\\Department of Mathematics - Hill Center\\110 Frelinghuysen Rd.\\Piscataway, NJ 08854}
%\curraddr{}
\email{duffyc@math.rutgers.edu}
\thanks{I would like to thank my advisor, Prof. Robert Wilson.}

%\subjclass[2000]{Primary }

%\keywords{}

\date{April 2008}

%\dedicatory{}

\begin{abstract}In this paper we will study the structure of
algebras $A(\Gamma)$ associated to two directed, layered graphs
$\Gamma$.  These are algebras associated with Hasse graphs of n-gons
and the algebras $Q_n$ related to pseudoroots of noncommutative
polynomials.  We will find the filtration preserving
automorphism group of these algebras and then we will find the
multiplicities of the irreducible representations of
$Aut(A(\Gamma))$ acting on the homogeneous components of $A(\Gamma)$
and $A(\Gamma)^!$.
\end{abstract}

\maketitle

\section*{Introduction}

In this paper we will be considering directed graphs $\Gamma$
satisfying certain hypotheses.  There exists an algebra $A(\Gamma)$
over a field $k$, an associated graded algebra gr$A(\Gamma)$, and a
dual algebra $A(\Gamma)^!$ of gr$A(\Gamma)$ associated with each of
these graphs. In Section \ref{S:prelim} we will give some
preliminaries on how these algebras are built from the graphs as
well as a basis for $A(\Gamma)$.  The definition of the dual algebra
and of subalgebras that will play an integral role in what follows
will be given in Section \ref{S:subalg}.  In Section \ref{S:def} we
will introduce the two algebras $A(\Gamma_{D_n})$ and $Q_n$ that we
will be describing throughout this paper.

The automorphism group of the graph injects into the automorphism
group Aut$A(\Gamma)$ of $A(\Gamma)$.  Furthermore, the nonzero
scalars inject into the automorphism group of the algebra.  Thus,
one is naturally led to ask how these automorphism groups are
related. This question will be answered in Section \ref{S:aut}.

A second question that we are led to is to describe the homogeneous
components of $A(\Gamma)$ and $A(\Gamma)^!$.  We will decompose the
graded algebra and its dual into irreducible
Aut$(A(\Gamma))$-modules by calculating the graded trace generating
functions.  These graded trace generating functions are actually the
graded dimensions of certain subalgebras of gr$A(\Gamma)$.  Hence,
the technique for calculating the graded trace generating functions
is to abstract the problem into finding the graded dimension of
subalgebras.  In Section \ref{S:grtr} we will explain why this is
true and how these graded dimensions are found.

The graded traces of our two particular algebras will be given in
Section \ref{S:genfn} and their decompositions in Section
\ref{S:reps}.  The decomposition can be found using the graded trace
and the characters of the automorphism group of the algebra. The
dual algebras of these two examples will be considered in Section
\ref{S:dual}.

\section{Preliminaries}\label{S:prelim}

\subsection{The Algebra $A(\Gamma)$}

Certain algebras, denoted $A(\Gamma)$, associated to directed graphs
were first defined by Gelfand, Retakh, Serconek, and Wilson
~\cite{GRSW}. We recall the definitions of $A(\Gamma)$ and
gr$A(\Gamma)$ following the development found in ~[\cite{RSW}, \S2].
Let k be a field and for any set $W$ let $T(W)$ be the free
associative algebra on $W$ over k. Let $\Gamma = (V,E)$ be a
directed graph where $V$ is the set of vertices, $E$ the set of
edges, and there are functions $t,h:E\rightarrow V$ (tail and head
of $e$). We say $\Gamma$ is a layered graph if $V=\displaystyle
\bigcup_{i=0}^nV_i,\, E= \displaystyle \bigcup_{i=0}^nE_i,\,
t:E_i\rightarrow V_i, \text{ and }h:E_i\rightarrow V_{i-1}$.  If
$v\in V_i$ ($e\in E_i$), we say the level of $v$, (respectively $e$)
is $i$; denote this by $|v|$, (resp. $|e|$).  We will assume
throughout this paper that $V_0=\{*\}$ and that for all $v\in V_+ =
\displaystyle \bigcup_{i=1}^n V_i$, there exists at least one $e\in
E$ such that $t(e)=v$.  For each $v\in V_+$, fix some $e_v\in E$
with $t(e_v)=v$; call this a distinguished edge.

A path from $v\in V \text{ to } w\in V$ is a sequence of edges $\pi
= \{e_1,...,e_m\}$ such that $t(e_1)=v,\, h(e_m)=w, \text{ and
}t(e_{i+1})=h(e_i), 1\leq i<m$.  We will say $t(\pi)=v, \,
h(\pi)=w$, and the length of $\pi, l(\pi), \text{ is }m$. Write
$v>w$ if there exists a path from $v$ to $w$.  For
$\pi=\{e_1,...,e_m\}$, define $e(\pi,k):=\displaystyle \sum_{1\leq
i_1<\cdots<i_k\leq m} e_{i_1}\cdots e_{i_k}$.  For each $v\in V$
there is a path $\pi_v=\{e_1,...,e_m\}$, called the distinguished
path, from $v$ to * defined by $e_1=e_v$, $e_{i+1}=e_{h(e_i)}$ for
$1\leq i<m$, and $h(e_m)=*$.  When $\pi_v$ is the distinguished path
from $v$ to
*, we will write $e(v,k)$ in lieu of $e(\pi_v,k)$.  Let $R$ be the
two-sided ideal of $T(E)$ generated by $\{e(\pi_1,k)-e(\pi_2,k):
t(\pi_1)=t(\pi_2),\, h(\pi_1)=h(\pi_2),\, 1\leq k\leq l(\pi_1)\}$.

\begin{definition}$A(\Gamma)=T(E)/R$\end{definition}

Let $\hat{e}(v,k)$ denote the image in $A(\Gamma)$ of $e_1\cdots e_k$.
Finally say that $(v,k)$ covers $(w,l)$ if $v>w$ and $k=|v|-|w|$, write this as $(v,k)\gtrdot (w,l)$.

\begin{thm}\label{T:rswbasis} ~[\cite{RSW},Thm 1] - Let $\Gamma=(V,E)$ be a layered graph,
$V=\displaystyle \bigcup_{i=0}^nV_i,\, V_0=\{*\}$.  Then\\
$\mathcal{B}(\Gamma) := \{\hat{e}(v_1,k_1) \cdots \hat{e}(v_l,k_l):
l\geq 0, v_1,...,v_l\in V_+, 1\leq k_i\leq |v_i|,
(v_i,k_i)\not\gtrdot (v_{i+1},k_{i+1})\}$ is a basis for
$A(\Gamma)$.\end{thm}

There is also a presentation of $A(\Gamma)$ as a quotient of
$T(V_+)$ ~[\cite{RSW2},\S3].  Every edge may be expressed as a
linear combination of distinguished edges, and the distinguished
edge $e_v$ may be identified with $v\in V_+$. Define $S_1(v):=\{w\in
V_{|v|-1} :v>w\}$.  A layered graph is uniform if for every $v\in
V_j, \, j\geq 2$, every pair of vertices $u,w$ in $S_1(v)$ satisfies
$S_1(u)\cap S_1(w) \neq \emptyset$ (``diamond condition'').

\begin{prop}\label{P:rsw2def} ~[\cite{RSW2},Prop 3.5] Let $\Gamma$ be a uniform layered graph.
 Then $A(\Gamma)\cong T(V_+)/R_V$ where $R_V$ is the two-sided ideal generated by

\noindent $\{v(u-w)-u^2+w^2 + (u-w)x : v\in \displaystyle \bigcup_{i=2}^n
V_i, u, w \in S_1(v), x\in S_1(u)\cap S_1(w)\}$.\end{prop}

\noindent Remark: From now on we will just write $e(v,k) \text{ for
} \hat{e}(v,k)$.

\subsection{Associated Graded Algebra $gr A(\Gamma)$}

Next we will describe a filtration and grading on $A(\Gamma)$. Here
we will also denote by $V$ the span of $V$ in $T(V)$, and by $E$ the
span of $E$ in $T(E)$, when no confusion will arise.  Let
$W=\displaystyle \sum_{k\geq 0}W_k$ be a graded vector space (in our
case $W=V \text{ or }E$).  Then $T(W)$ is bigraded.  One grading
$T(W) = \sum_iT(W)_{[i]}$ is given by degree in the tensor algebra;
i.e., $T(W)_{[i]} = span\{w_1\cdots w_i : w_1,...,w_i\in W\}$.  The
other grading is given by $T(W)=\displaystyle \sum_{i\geq0}T(W)_i$
where

\noindent $T(W)_i=\text{span}\{w_1\cdots w_r:r\geq 0, w_j\in W_{l_j},
l_1+...+l_r=i\}$.

\noindent The second grading induces an increasing
filtration on $T(W)$:

\noindent $T(W)_{(i)}=\text{span}\{w_1\cdots w_r:r\geq 0, w_j\in
W_{l_j}, l_1+\cdots+l_r\leq i\} = T(W)_0+\cdots+T(W)_i$.

Because $T(W)_{(i)}/T(W)_{(i-1)} \cong T(W)_i$, $T(W)$ can be
identified with its associated graded algebra.  Define a map
$gr:T(W)\backslash\{0\}\rightarrow T(W)\backslash\{0\}=grT(W)$ by
$w=\displaystyle \sum_{i=0}^k w_i \mapsto w_k$ where $w_i\in
T(W)_i,\, w_k\neq 0$.  Of course, gr is not an additive map.  ~[\cite{RSW2},\S2]

\begin{lemma}~[\cite{RSW2},Lemma 2.1] Let W be a graded vector space and I an ideal in $T(W)$.
Then \\$gr(T(W)/(I)) \cong T(W)/(grI)$.\end{lemma}

Thus the associated graded algebra of $A(\Gamma), \, grA(\Gamma)$,
is isomorphic to $T(E)/grR$.  The graded relations, $grR$, are for
$\pi_1 = \{e_1,...,e_m\} \text{ and }\pi_2=\{f_1,...,f_m\}$,
$\{e_1\cdots e_k=f_1\cdots f_k, \, 1\leq k\leq m\}$ (the leading
term of $e(v,k)$).  Another way to consider this is that $e(v,k+l) =
e(v,k)e(u,l)$ in $grR$ where $v\gtrdot u$.  Recalling the definition
of $\mathcal{B}(\Gamma)$ from Theorem \ref{T:rswbasis}, we see that
$\{gr (b) :b\in \mathcal{B}(\Gamma)\}$ is a basis for $grA(\Gamma)$.

Let us now look at our second description of $A(\Gamma)$ as
isomorphic to $T(V_+)/R_V$.

\begin{prop}\label{P:rsw2grdef}~[\cite{RSW2},Prop 3.6] Let $\Gamma$ be a
uniform layered graph.  \\Then $grA(\Gamma) \cong T(V_+)/grR_V$ where
$gr R_V$ is generated by $\{v(u-w): v\in \displaystyle
\bigcup_{i=2}^n V_i, u, w \in S_1(v)\}$.\end{prop}

Also, $A(\Gamma)_i = (T(E)_i+R)/R = (T(V_+)_i+R_V)/R_V$ and
$A(\Gamma)_{(i)} = (T(E)_{(i)}+R)/R = \\(T(V_+)_{(i)}+R_V)/R_V$.

\section{The Dual $A(\Gamma)^!$ and the Subalgebra $A(\Gamma^\sigma)$}\label{S:subalg}

\subsection{The Dual $A(\Gamma)^{!}$}

\begin{definition}[$A^!$] [\cite{BVW},$\S$ 2] Let $A=T(E)/(R), \, R\subseteq
E^{\otimes 2}$.  Then $A^! = T(E^*)/(R^\perp)$ where $E^*$ is the
dual vector space of $E$ and $R^\perp$ is the annihilator of $R$;
i.e. $R^\perp = \{f\in (E^{\otimes 2})^*:f(x)=0 \, \forall x\in R\}$
of $(E^{\otimes 2})^*$ where $(E^{\otimes 2})^*$ is canonically
identified with $E^{*\otimes 2}$.\end{definition}

\begin{definition}[$A(\Gamma)^{!}$]~\cite{D} The dual of gr$A(\Gamma)$ is $A(\Gamma)^{!} :=
T(E^*)/(gr R)^\perp$.
\end{definition}

\noindent The dual element to the generator $e(v,k)$ in $A(\Gamma)$
will be denoted $e(v,k)^*$.

\begin{prop}\cite{D}$A(\Gamma)^{!}$ has a presentation with generators
$\{e(v,1)^*\}$ and relations \\$\{e(v,1)^*e(u,1)^*: v\not\gtrdot
u\}\cup \{e(v,1)^*\displaystyle \sum_{v\gtrdot
u}e(u,1)^*\}$\end{prop}

\subsection{The Subalgebra $A(\Gamma^\sigma)$}

We will now define a subalgebra of $grA(\Gamma)$.  Let $\sigma$ be
an automorphism of the layered graph $\Gamma$; i.e. an automorphism
that preserves each level of the graph.  Define
$\Gamma^\sigma:=(V_\sigma,E_\sigma)$ where $V_\sigma$ is the set of
vertices $v\in V$ such that $\sigma(v)=v$ and $E_\sigma$ is the set
of edges that connect the vertices minimally.  Here minimally means
that there is an edge $e\in E_\sigma$ from $v$ to $w$, $v,\,w\in
V_\sigma$ if and only if $v\geq u\geq w$, $u\in V_\sigma$, implies
$u=v$ or $u=w$.

\begin{definition}[$A(\Gamma^\sigma)$]Define $A(\Gamma^\sigma)$ to be
span$\{e(v_1,k_1)\cdots e(v_l,k_l): l\geq 0, v_1,...,v_l\in
V_\sigma\backslash *, \\1\leq k_i\leq |v_i|, (v_i,k_i) \not\gtrdot
(v_{i+1},k_{i+1})\}$.\end{definition}

This set is, in fact, a basis.  The elements are linearly
independent because the set is a subset of a basis.

\begin{thm} $A(\Gamma^\sigma)$ is a subalgebra of gr$A(\Gamma)$.  \\A
presentation for $\ags$ is given by generators  $G'=\{e'(v,k): v\in
V_\sigma, 1\leq k \leq |v|\}$ and relations $R'=\{e'(v,k+l) -
e'(v,k)e'(u,l): v\gtrdot u\in V_\sigma\}$.\end{thm}

\begin{proof}Define $\phi:T(G') \rightarrow grA(\Gamma)$ by
$\phi(e'(v,k))=e(v,k)$.  We have $\phi(T(G'))\supseteq \ags$ because
elements of $B_\sigma$ are formed from products of elements in $G'$.

In $A(\Gamma)$ we have \[(*) \mbox{\hspace{.2in}} e(v,k)e(u,l)-e(v,k+l)\equiv
\displaystyle\sum_{\substack{i_0,i_{r+1} \geq 0, i_1,...,i_r
\geq1\\
i_0<k, i_0+\cdots +i_r \leq k\\ i_0+\cdots+i_{r+1}=k+l}}
(-1)^{r+1}e(v,i_0)e(u,i_1)\cdots e(u,i_{r+1})\] mod R ~[\cite{GRSW},
p6]. However, the elements on the right-hand side are all of lower
degree than those on the left-hand side ~[\cite{GRSW}, Lemma 2.2].
Note that the elements on the left-hand side have degree k+l and are
in $(T(E)/R)_{[(k+l)|v|-(k+l)(k+l+1)/2]}$. Therefore, in $gr
A(\Gamma)$, the terms on the right-hand side are zero. Hence, we
have $e(v,k)e(u,l)-e(v,k+l)\equiv 0$.  Consequently
$\phi(e'(v,k)e'(u,l)) = \phi(e'(v,k+l))$.

Let $b'=e'(v_1,k_1)\cdots e'(v_l,k_l)$ be a monomial in $T(G')$.  In
$T(G')/<R'>$, we may replace every occurrence of
$e'(v_i,k_i)e'(v_{i+1},k_{i+1})$ such that $(v_i,k_i)\gtrdot
(v_{i+1},k_{i+1})$ in $b'$ with $e'(v_i,k_i+k_{i+1})$. Thus $b'
\equiv e'(v_1',k_1')\cdots e'(v_l',k_l')$ such that
$(v_i',k_i')\not\gtrdot (v_{i+1}',k_{i+1}')$ in $T(G')/<R'>$.
Hence $\phi(b')\in \ags$, and so $\phi(T(G'))=\ags$.

By (*), $R'\subseteq \text{ker}\phi$ and we have an induced
surjective homomorphism $\phi':T(G')/<R'> \rightarrow \ags$.

Let $f=\sum k_ib_i' \in \text{ker}\phi'$, where $k_i$ is an element
in the field and $b_i'$ a monomial in \mbox{$T(G')/<R'>$.} Then
$0=\phi'(f)=\sum k_i\phi'(b_i') = \sum k_ib_i$ is a linear
combination of basis elements in $\ags$. This implies that $k_i=0
\forall i$ and so $f=0$. Therefore, $\phi'$ is an isomorphism.
\end{proof}

We will write $e(v,k)$ for $e'(v,k)$ from now on.

\bigskip

For $x$ a basis element of $A(\Gamma)_{[i]}$, write $\sigma(x)$ as a
linear combination of basis elements and say the coefficient of $x$
in $\sigma(x)$ is $\alpha$.  Denote this value $\alpha$ by
$t_\sigma(x)$. Then, for finite-dimensional $A(\Gamma)_{[i]}$,
$Tr_\sigma(A(\Gamma)_{[i]}) = \sum_{x\in \text{basis}}t_\sigma(x)$.
In this paper we will be looking at the trace of $\sigma$ acting on
$A(\Gamma)_{[i]}$ and $A(\Gamma)_{[i]}^{!}$.

\section{Definition and Hilbert Series of Two Algebras}\label{S:def}

\subsection{Hasse graph of an n-gon: $\Gamma_{D_n}$}

A Hasse graph, or Hasse diagram, is a graph which represents a
finite poset $\mathcal{P}$.  The vertices in the graph are elements
of $\mathcal{P}$ and there is an edge between $x,\,y\in \mathcal{P}$
if $x<y$ and there does not exist a $z\in\mathcal{P}$ such that
$x<z<y$.  Furthermore, the vertex for $x$, $v_x$, is in a lower level than that for $y$, $v_y$ (if we talk about layers in the
graph, $|v_x|=|v_y|-1$).

Consider a polytope.  We can put a partial order on the set of
k-faces in the polytope by $x<y$ if $x$ is an ($n-1$)-face, $y$ is
an $n$-face and $x$ is a face of $y$.

Thus, the Hasse graph of an n-gon has one vertex in levels 0 and 3
and n vertices on levels 1 and 2.  The top vertex is connected to
all vertices in level 2 (all edges are in the 2-dimensional
polygon), each vertex in level 2 is connected to the vertex directly
below it and the one to that vertex's right, with wrapping around to
the first vertex in level one for the last vertex in level two (each
edge connects two adjacent vertices), and each vertex in level 1 is
connected to the minimal vertex. Label the vertices by using
subscripts in $\mathbb{Z}/(n)$.  In level 1 call the vertices
$w_1,...,w_n$, call the vertices in level 2 $v_{12},...,v_{n1}$
(where the subscripts indicate to which vertices in level 1 the
vertex is connected), and the top vertex is $u$. See Figure
\ref{Fi:dn}.

\medskip

\begin{figure}[h]
\begin{center}
\includegraphics[height=1in]{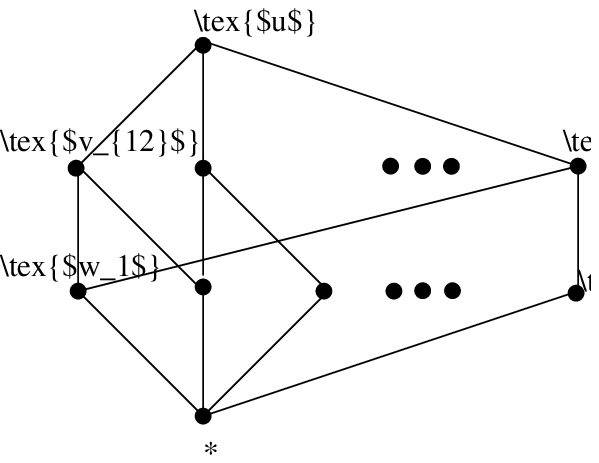}
\end{center}
\caption{$\Gamma_{D_n}$\label{Fi:dn}}
\end{figure}

\bigskip

We consider the algebra $A(\Gamma_{D_n})$ determined by this graph.
The construction of this algebra is described in ~\cite{GRSW} (see
$\S$ \ref{S:prelim}). In brief (using the definition given in
Proposition \ref{P:rsw2def}), the generators are the vertices and
the relations are that two paths which have the same starting and
ending vertices are equivalent.  These relations are, for $1\leq
i\leq n$,

\noindent 1) $v_{ii+1}(w_i-w_{i+1})-w_i^2+w_{i+1}^2$

\noindent 2) $u(v_{i i+1}-v_{i+1 i+2})-v_{i i+1}^2+v_{i+1
i+2}^2+(v_{i i+1}-v_{i+1 i+2})w_{i+1}$.

\bigskip

We will give here two bases for $A(\Gamma_{D_n})$, one for each of
the two definitions of $A(\Gamma)$ given in Section \ref{S:prelim}.
First we will give a basis in terms of the vertices (Proposition
\ref{P:rsw2def}).

\begin{prop}\label{P:diamond}A basis $\mathcal{B}$ of
$A(\Gamma_{D_n})$ consists of * and the set of all words in $u, \,
v_{i\,i+1}, \text{ and } w_i$ such that the following conditions on
the words hold: the subword $v_{ii+1}w_j$ only occurs if $j\neq
i+1$, the subword $uv_{ii+1}$ only if $i=1$, and $uv_{ii+1}w_j$ only
if $i=j=1$.\end{prop}

\begin{proof}
We can describe a basis of monomials for $A(\Gamma_{D_n})$ using
Bergman's Diamond Lemma ~\cite{Berg}. Put a partial order on the
generators such that $u>v_{i\,i+1}>w_j \forall i, j$ and
$v_{i\,i+1}>v_{j\,j+1} \text{ and } w_i>w_j \text{ if } i>j$. Order
monomials lexicographically.  The reductions are $uv_{i+1\,i+2}
\equiv
uv_{i\,i+1}-v_{ii+1}^2+v_{i+1i+2}^2+(v_{ii+1}-v_{i+1i+2})w_{i+1}, \,
1\leq i\leq n-1$ and $v_{ii+1}w_{i+1}\equiv
v_{ii+1}w_{i}-w_i^2+w_{i+1}^2, \, 1\leq i\leq n$.

We need to find a complete list of reductions so that all
ambiguities resolve.  The only ambiguity will occur when we have a
word that ends in $v$ overlapping with one beginning with $v$; i.e.
$uv_{i+1\,i+2}w_{i+2}$.

\noindent $(uv_{i+1\,i+2})w_{i+2} \equiv
(uv_{ii+1}-v_{ii+1}^2+v_{i+1i+2}^2+(v_{ii+1}-v_{i+1i+2})w_{i+1})w_{i+2}
\equiv
uv_{ii+1}w_{i+2}-v_{ii+1}^2w_{i+2}+v_{i+1i+2}(v_{i+1i+2}w_{i+1}-w_{i+1}^2+w_{i+2}^2)+(v_{ii+1}w_i-w_i^2
+w_{i+1}^2)w_{i+2}-v_{i+1i+2}w_{i+1}w_{i+2} \equiv
uv_{ii+1}w_{i+2}-v_{ii+1}^2w_{i+2}+v_{i+1i+2}^2w_{i+1}
-v_{i+1i+2}w_{i+1}^2+v_{i+1i+2}w_{i+1}w_{i+2}-w_{i+1}^2w_{i+2}+w_{i+2}^3+v_{ii+1}w_iw_{i+2}-w_i^2w_{i+2}
+w_{i+1}^2w_{i+2}-v_{i+1i+2}w_{i+1}w_{i+2} =
uv_{ii+1}w_{i+2}+v_{i+1i+2}^2w_{i+1}-v_{ii+1}^2w_{i+2}
-v_{i+1i+2}w_{i+1}^2+v_{ii+1}w_iw_{i+2}+w_{i+2}^3-w_i^2w_{i+2}$

and

\noindent $u(v_{i+1\,i+2}w_{i+2}) \equiv
u(v_{i+1i+2}w_{i+1}-w_{i+1}^2+w_{i+2}^2) \equiv (uv_{ii+1}-
v_{ii+1}^2+v_{i+1i+2}^2+(v_{ii+1}-v_{i+1i+2})w_{i+1})w_{i+1}\\
-uw_{i+1}^2+ uw_{i+2}^2 \equiv uv_{ii+1}w_i-uw_i^2+ uw_{i+1}^2-
v_{ii+1}(v_{ii+1}w_i-w_i^2+ w_{i+1}^2)+ v_{i+1i+2}^2w_{i+1} +
v_{ii+1}w_iw_{i+1}- w_i^2w_{i+1}+ w_{i+1}^3- v_{i+1i+2}w_{i+1}^2-
uw_{i+1}^2+ uw_{i+2}^2 \equiv uv_{ii+1}w_i+ uw_{i+2}^2- uw_i^2-
v_{ii+1}^2w_i+ v_{ii+1}w_i^2-v_{ii+1}w_{i}w_{i+1}+w_i^2w_{i+1}
-w_{i+1}^3+v_{i+1i+2}^2w_{i+1}+v_{ii+1}w_iw_{i+1}-w_i^2w_{i+1}+w_{i+1}^3-v_{i+1i+2}w_{i+1}^2
=
uv_{ii+1}w_i+uw_{i+2}^2-uw_i^2+v_{i+1i+2}^2w_{i+1}-v_{ii+1}^2w_i-v_{i+1i+2}w_{i+1}^2+v_{ii+1}w_i^2$.

Thus we need to add an additional reduction; namely,
$uv_{ii+1}w_{i+2}\equiv uv_{ii+1}w_i+uw_{i+2}^2-
uw_i^2+v_{ii+1}^2w_{i+2}-v_{ii+1}^2w_i-v_{ii+1}w_iw_{i+2}+v_{ii+1}w_i^2-w_{i+2}^3+w_i^2w_{i+2}$.
This does not create additional ambiguities since this reduction
ends in $w$ and we have no reductions which begin in $w$.  Also, no
reductions end in u.  Thus, all ambiguities now resolve.

Therefore, by Bergman's Diamond Lemma, $A(\Gamma_{D_n})$ may be
identified with the k-module of monomials which are irreducible
under these reductions.  Hence, $\mathcal{B}$ is a basis for
$A(\Gamma_{D_n})$.
\end{proof}

Next follows a basis in terms of edges (Thm \ref{T:rswbasis}).

\begin{prop}\label{P:ebasis}$\mathcal{B}' = \{e(x_1,k_1) \cdots e(x_l,k_l): l\geq 0,
x_1,...,x_l\in \{u,v_{12},...,v_{n1},w_1,...,w_n\}, \\1\leq k_i\leq
|x_i|,
(x_i,k_i)\not\gtrdot (x_{i+1},k_{i+1})\}$ is a basis for
$A(\Gamma_{D_n})$.
\end{prop}

\begin{proof}This follows directly from Theorem \ref{T:rswbasis}.\end{proof}

\noindent In the preliminaries we stated that the algebra is
generated by distinguished edges and so we can identify the
distinguished edges with the vertices which are their tails - $e_v$
is identified with $v$.  Thus $e(v,k)$ can be expressed as a product
of k vertices (recall we are writing $e(v,k)$ in lieu of
$\hat{e}(v,k)$), and so there is a correlation between the bases
$\mathcal{B},\, \mathcal{B}'$ as follows:

$e(u,3) \leftrightarrow uv_{12}w_1 \qquad e(u,2) \leftrightarrow
uv_{12} \qquad e(u,1) \leftrightarrow u \\e(v_{ii+1},2)
\leftrightarrow v_{ii+1}w_i \qquad e(v_{ii+1},1) \leftrightarrow
v_{ii+1} \qquad e(w_i,1) \leftrightarrow w_i$
%mention identification in gr alg?

\noindent It is important to observe that in the associated graded
algebra, $\sigma \in Aut(A(\Gamma))$ permutes the elements of
gr$\mathcal{B}$ and gr$\mathcal{B}'$.

%show bases are permuted by sigma in asso gr alg?
\medskip
Recall that the Hilbert series gives the graded dimension of an
algebra; the coefficient of $t^k$ is the k-th graded dimension (see
Section \ref{S:prelim} for the grading on our algebras).  We write
this as $H(t)=\sum dim(A(\Gamma)_{[k]})t^k$.

\begin{prop}The Hilbert series for $A(\Gamma_{D_n})$ is
$$H(t)=\frac{1}{1-(2n+1)t+(2n-1)t^2-t^3}=\frac{1-t}{1-(2n+2)t+4nt^2-2nt^3+t^4}$$\end{prop}

\begin{proof}
We will give two proofs of this proposition.  The first one uses
Proposition \ref{P:diamond} and induction to count basis elements.
This will give us a recursion that can then be written as a
generating function.  The second method of proof is much shorter and
uses a theorem from ~\cite{RSW} that gives a formula for the Hilbert
series of an algebra associated to a directed, layered graph.

\noindent \textbf{Method 1:} By Proposition \ref{P:diamond}, there
are $n(n-1)$ subwords of the form $v_{ii+1}w_{j}$ which can occur in
an element of $\mathcal{B}$ and exactly one of the forms
$uv_{i\,i+1}$ and $uv_{i\,i+1}w_j$. This means that
there are $n$ subwords of the form $v_{ii+1}w_{j}$ which cannot
occur, $n-1$ of the form $uv_{i\,i+1}$, and $n^2-1$ of the form
$uv_{i\,i+1}w_j$.

Let $d_k=\text{dim}(A(\Gamma_{D_n})_{[k]})$.

\noindent We will proceed by induction.

\begin{itemize}
\item $d_0=1$

\item $d_1=2n+1$: Every word of length one belongs to the basis since
 all reducible subwords are of length greater than one. A basis is:
 $\{u$, $v_{ii+1}$, $w_i: 1\leq i\leq n\}$

\item $d_2=4n^2+2n+2$: There are $(2n+1)^2$ elements of length two and
$2n-1$ of them are reducible.  Hence, the dimension is
$(2n+1)^2-(2n-1)=4n^2+2n+2$. A basis is: $\{w_iu$, $w_iv_{j\,j+1}$,
$w_iw_j$, $uu$, $uv_{1\,2}$, $uw_i$, $v_{ii+1}u$,
$v_{ii+1}v_{jj+1}$, $v_{ii+1}w_j (j\neq i+1): 1\leq i,j\leq n\}$
\end{itemize}

\noindent Use induction to determine $d_k$: \begin{itemize}

\item If $x\in\mathcal{B}$ is a word of length $k-1$, then $w_ix\in
\mathcal{B}$.  Thus there are $nd_{k-1}$ words of length $k$ in
$\mathcal{B}$ starting with $w_i$.

\item If $x\in \mathcal{B}$ is a word of length $k-1$, then
$v_{ii+1}x\in \mathcal{B}$ if and only if $x$ does not begin with
$w_{i+1}$. As determined in the previous bullet, there are
$nd_{k-2}$ basis elements starting with $w_j, 1\leq j\leq n$ in
degree $k-1$, and thus $d_{k-2}$ of them beginning with $w_{i+1}$.
Hence, for each $i$, there are $d_{k-1}-d_{k-2}$ possibilities for
$x$. Therefore, there are $n(d_{k-1}-d_{k-2})$ words of length $k$
of the form $v_{i\,i+1}x$.

\item We will treat the case of words beginning with $u$ in three
cases.  If $x\in \mathcal{B}$ of length $k-2$, $uux\in \mathcal{B}$
if and only if $x$ does not begin with $v_{ii+1}, 2\leq i\leq n$.
There are $d_{k-1}-nd_{k-2}-n(d_{k-2}-d_{k-3})$ words beginning with
$u$ in degree $k-1$ (from previous bullets).  Thus, there are that
many words of the form $uux\in \mathcal{B}$.  Next $uv_{12}x\in
\mathcal{B}$ if and only if $x$ does not begin with $w_i, 2\leq
i\leq n$.  Thus, there are $d_{k-2}-(n-1)d_{k-3}$ words of the form
$uv_{12}x$.  Finally, $uw_ix\in \mathcal{B}$ for all $x$.  Thus,
there are $nd_{k-2}$ words of this form.  This gives us a total of
$d_{k-1}-2nd_{k-2}+nd_{k-3}+d_{k-2}-(n-1)d_{k-3}+nd_{k-2} =
d_{k-1}-(n-1)d_{k-2}+d_{k-3}$ words beginning with $u$.

\end{itemize}

\noindent Thus, $d_k =
nd_{k-1}+n(d_{k-1}-d_{k-2})+d_{k-1}-(n-1)d_{k-2}+d_{k-3} =
(2n+1)d_{k-1}-(2n-1)d_{k-2}+d_{k-3}$.

We can write this recurrence formula as a generating function
following the method described by Wilf in ~[\cite{Wilf},\S1.2].  Let
$H(t)=\sum_{i\geq0}d_it^i$ denote the generating function that we
are trying to find.  Let $d_{-2}=d_{-1}=0, d_0=1$.  Multiply both
sides of the recursion by $t^i$ and sum over $i\geq0$.  Then on the
left-hand side we have $d_1+d_2t+d_3t^2+...=\frac{H(t)-d_0}{t}$. And
on the right hand side we have $(2n+1)H(t)-(2n-1)tH(t)+t^2H(t)$.
Solving for $H(t)$:

$H(t)-1=H(t)[(2n+1)t-(2n-1)t^2+t^3] \Rightarrow$

$H(t)[1-(2n+1)t+(2n-1)t^2-t^3]=1 \Rightarrow$

$H(t)=\frac{1}{1-(2n+1)t+(2n-1)t^2-t^3}$.

\medskip

\noindent \textbf{Method 2:}  ~[\cite{RSW},Thm 2] gives the Hilbert
series formula as:
\begin{equation}\label{Eq:Hilbert}H(A(\Gamma),t)=\frac{1-t}{1+\displaystyle \sum_{v_1>\cdots>v_l\geq *}
(-1)^lt^{|v_1|-|v_l|+1}}.\end{equation}

In this example, the possible sequences indexing the sum are: $u,
v_{ii+1},w_i,*, u>v_{ii+1}, v_{ii+1}>w_i, v_{ii+1}>w_{i+1}, w_i>*,
u>w_i, u>v_{ii+1}>w_i, u>v_{ii+1}>w_{i+1}, v_{ii+1}>*,
v_{ii+1}>w_i>*, v_{ii+1}>w_{i+1}>*, u>*, u>v_{ii+1}>*, u>w_i>*,
u>v_{ii+1}>w_{i}>*, \text{ and } u>v_{ii+1}>w_{i+1}>*$.  Thus, the
coefficients of $t,\, t^2,\, t^3,\text{ and } t^4$ are $-(2n+2),\,
n+2n+n=4n, \, n+n-2n-2n=-2n,\text{ and } 1-2n+2n=1$, respectively.
The coefficient of $t^k$ for $k\geq 5$ is zero.
Thus \[H(A(\Gamma_{D_n}),t)=\frac{1-t}{1-(2n+2)t+4nt^2-2nt^3+t^4}. \qedhere\]
\end{proof}

\subsection{The Algebra $Q_n$}

The algebras $Q_n$ are the algebras associated with the lattice of
subsets of $\{1,2,...,n\}$.  Label the vertices in level $i$ by
$\{v_A: A\subseteq \{1,...,n\}, |A|=i\}$.  $Q_4$ is shown in Figure
\ref{Fi:q4} below.  Their history and some properties are discussed
in ~\cite{GRSW}.

\medskip
\begin{figure}[h]
\begin{center}
\includegraphics[height=1in]{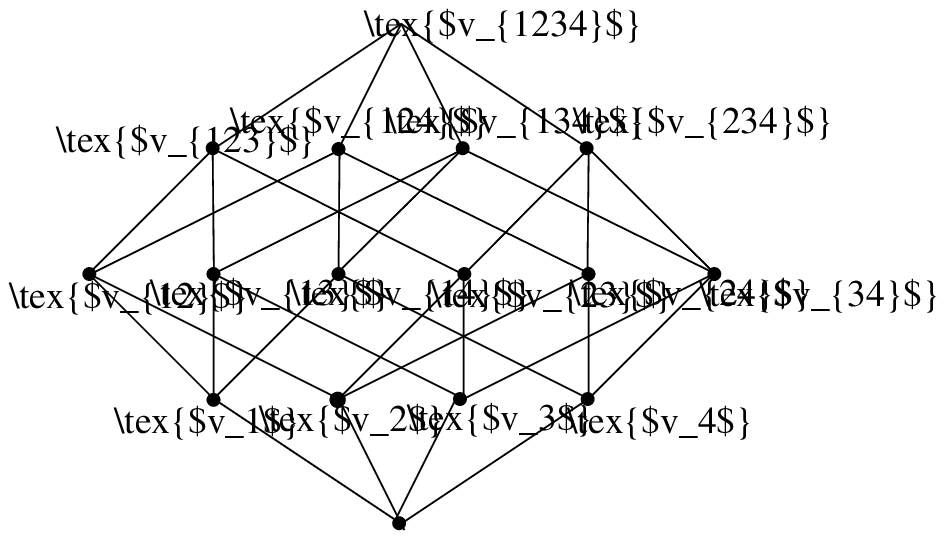}
\end{center}
\caption{$Q_4$\label{Fi:q4}}
\end{figure}
\medskip

\newpage

Following the definition given in Proposition \ref{P:rsw2def}, the
generators of $Q_n$ are the vertices $\{v_A:A\subseteq
\{1,...,n\}\}$ and the relations are

\[(*) \{v_{A}(v_{A\backslash i}- v_{A\backslash j})- v_{A\backslash
i}^2 +v_{A\backslash j}^2 + (v_{A\backslash i}-v_{A\backslash j})
v_{A\backslash\{i,j\}}: A\subseteq \{1,...,n\}, i,j\in A\}.\]

\noindent Furthermore,
\begin{prop}\label{P:qnebasis}$\mathcal{B_Q} =
\{e(v_{A_1},k_1) \cdots e(v_{A_l},k_l): l\geq 0,
A_1,...,A_l\subseteq \{1,...,n\}, 1\leq k_i\leq |v_{A_i}|=|A_i|,
(v_{A_i},k_i)\not\gtrdot (v_{A_{i+1}},k_{i+1})\}$ is a basis for
$Q_n$.
\end{prop}

\begin{proof}This follows directly from Theorem \ref{T:rswbasis}.\end{proof}

\medskip

In ~[\cite{RSW},Thm 3], Retakh, Serconek, and Wilson prove that
$$H(Q_n,t)=\frac{1-t}{1-t(2-t)^n}$$ using Equation \ref{Eq:Hilbert}.

\section{Automorphism Groups of the Algebras}\label{S:aut}

Throughout this paper Aut($A$) will denote the filtration-preserving
automorphisms of the graded algebra $A$ (see $\S$ \ref{S:prelim}).

\begin{lemma}\label{L:grpaut}$Aut(A(\Gamma)) \supseteq k^*\times \text{Aut}(\Gamma)$\end{lemma}

\begin{proof}Any automorphism $\tilde{\sigma}$ of the graph extends
to an automorphism $\sigma$ of $T(E)$.  Since $\tilde{\sigma}$
preserves paths, $\sigma$ preserves the ideal $R$ defined in Section
\ref{S:prelim}. Hence it induces an automorphism, again denoted by
$\sigma$, of $A(\Gamma)=T(E)/R$. Also, for any scalar $\alpha$,
multiplication by $\alpha$ is an automorphism because the relations
are homogeneous. Thus $Aut(A(\Gamma)) \supseteq k^*\times
\text{Aut}(\Gamma)$. \end{proof}

Let $\Gamma$ be a graph with a unique minimal vertex at level 0,
$V_0=\{*\}$, whose vertices are labeled in the following manner.
Label the vertices in level one by $\{v_1,...,v_m\}$ and index those
in level $r,\, 2 \leq r \leq n$, by a subset of the power set of
$\{1,...,m\}$.  Let the edges connect vertices by minimal
containment of their indices.  There is a path from $v_A$
($|v_A|>|v_i|$) to $v_i$ if and only if $i\in A$.

\begin{thm}\label{T:autgen} Let $\Gamma$ be a graph as described
above. If $\Gamma$ satisfies i) $|V_1|>2$, ii) no two vertices have
the same label, and iii) there are either zero or two paths between
any two vertices which are two levels apart in $\Gamma$, then
Aut($A(\Gamma))=k^*\times \text{Aut}(\Gamma)$, k the base field.
\end{thm}

\begin{proof}  By Lemma \ref{L:grpaut}, $Aut(A(\Gamma)) \supseteq k^*\times \text{Aut}(\Gamma)$.  Any
automorphism of the algebra must preserve the relations.  Thus, for
all subsets $B,C\subseteq A \subseteq \{1,...,m\} \text{ such that
}v_A,v_B,v_C\in \Gamma \text{, }|v_A|\geq 2$, and
$|v_A|-|v_B|=|v_A|-|v_C|=1, |v_A|-|v_{B\cap C}|=2$ (i.e. $v_A, v_B,
v_C, v_{B\cap C}$ form a diamond), the image of $v_{A}(v_{B}-
v_{C})- v_{B}^2 +v_{C}^2 + (v_{B}-v_{C}) v_{B\cap C}$ must equal
zero. Consider first paths from level 2 to level 0. Because of our
assumption that there are exactly two paths from each vertex in
level two, $|A|=2$ for $v_A\in V_2$. Let $\sigma \in Aut(A(\Gamma))$
and $v_{A_1},...,v_{A_{m_2}}$ be vertices in level 2; then
$\sigma(v_{ij})=a_{A_1}^{ij}v_{A_1}+\cdots+a_{A_{m_2}}^{ij}v_{A_{m_2}}+b_1^{ij}v_1
+\cdots+b_m^{ij}v_m$ for all $i,j$ such that $v_{ij}\in \Gamma$ and
$\sigma(v_i)=c_1^iv_1+\cdots+c_m^iv_m$ for all $i$, where all
coefficients are in k.

Now $\sigma(v_{ij}(v_i-v_j))=\sigma(v_i^2-v_j^2)$ implies

\noindent $(a_{A_1}^{ij}v_{A_1}+\cdots
+a_{A_{m_2}}^{ij}v_{A_{m_2}}+b_1^{ij}v_1 +\cdots+b_m^{ij}v_m)
((c_1^i-c_1^j)v_1+\cdots+(c_m^i-c_m^j)v_m)$

\noindent
$=(c_1^iv_1+\cdots+c_m^iv_m)^2-(c_1^jv_1+\cdots+c_m^jv_m)^2.$

\noindent There are no $v_A's$ with $|A|=2$ on the right-hand side,
and so we must use our relations to eliminate them from the
left-hand side. Thus, every occurrence of $v_{kl}$ must be followed
by $v_k-v_l$; and hence, $c_k^i-c_k^j=-(c_l^i-c_l^j)$ and
$c_m^i-c_m^j=0$ if $m\neq k,l$.  Therefore, if $a_{kl}^{ij} \neq 0$,

\begin{equation}\label{Eq:ws2}(c_1^i-c_1^j)v_1+\cdots+(c_m^i-c_m^j)v_m =
(c_k^i-c_k^j)(v_k-v_l).\end{equation}

This has two consequences.  First, at most one $a_{kl}^{ij}$ can be
nonzero.  If all $a_{kl}^{ij}$ were zero, the element $v_{i\, j}
\notin (A(\Gamma))_{(1)}$ would be sent to an element in
$(A(\Gamma))_{(1)}$, which we cannot allow because then $\sigma$
would not be invertible. Thus $a_{kl}^{ij}$ must be nonzero for
exactly one $\{kl\}$.  Let us denote this set by
$\{\tau(i)\tau(j)\}$. Then $\sigma(v_{ij}) = a_{\tau(i)\tau(j)}^{ij}
v_{\tau(i)\,\tau(j)}+b_1^{ij}v_1+\cdots +b_m^{ij}v_m$.  If
$\tau(ij)=\tau(kl)$, then
$\sigma(v_{i\,j}-a_{\tau(i)\tau(j)}^{ij}(a_{\tau(k)\tau(l)}^{kl})^{-1}v_{k\,l})
\in (A(\Gamma))_{(1)}$; this implies that $\{ij\}=\{kl\}$. Thus
$\tau$ is one-to-one, and so is in $S_n$.

A second consequence of (\ref{Eq:ws2}) is that $c_r^i-c_r^j$ is zero
if and only if $r\neq \tau(i), \tau(j)$.

We now have from (\ref{Eq:ws2}):
$(a_{\tau(ij)}^{ij}v_{\tau(i)\,\tau(j)}+b_1^{ij}v_1+\cdots+b_m^{ij}v_m)
(c_{\tau(i)}^i-c_{\tau(i)}^j)(v_{\tau(i)}-v_{\tau(j)}) =\\
(c_1^iv_1+\cdots +c_m^iv_m)^2-(c_1^jv_1+\cdots+c_m^jv_m)^2$. (Recall
$c_{\tau(i)}^i-c_{\tau(i)}^j = -(c_{\tau(j)}^i-c_{\tau(j)}^j)$.)

Let $z=\displaystyle\sum_{r\neq \tau(i), \tau(j)} c_r^iv_r =
\displaystyle\sum_{r\neq \tau(i), \tau(j)} c_r^jv_r$.  Then
\begin{align}\label{Eq:zinreln}&(a_{\tau(ij)}^{ij}v_{\tau(i)\,\tau(j)}+b_1^{ij}v_1+\cdots+b_m^{ij}v_m)
((c_{\tau(i)}^i-c_{\tau(i)}^j)(v_{\tau(i)}-v_{\tau(j)})\\ &=
(c_{\tau(i)}^iv_{\tau(i)}+c_{\tau(j)}^{i}v_{\tau(j)}+z)^2-
(c_{\tau(i)}^jv_{\tau(i)}+c_{\tau(j)}^jv_{\tau(j)}+z)^2\nonumber\\
&= (c_{\tau(i)}^iv_{\tau(i)}+c_{\tau(j)}^{i}v_{\tau(j)})^2-
(c_{\tau(i)}^jv_{\tau(i)}+c_{\tau(j)}^jv_{\tau(j)})^2+
(c_{\tau(i)}^iv_{\tau(i)}+c_{\tau(j)}^{i}v_{\tau(j)})z\nonumber\\ &+
z(c_{\tau(i)}^iv_{\tau(i)}+c_{\tau(j)}^{i}v_{\tau(j)})-
(c_{\tau(i)}^jv_{\tau(i)}+c_{\tau(j)}^jv_{\tau(j)})z-
z(c_{\tau(i)}^jv_{\tau(i)}+c_{\tau(j)}^jv_{\tau(j)}).\nonumber\end{align}

On the left-hand side of (\ref{Eq:zinreln}), $v_r$, for $r\neq
\tau(i), \tau(j)$, is never the second term of the product of two
$v_r's$.  Hence, $(c_{\tau(i)}^iv_{\tau(i)}+
c_{\tau(j)}^{i}v_{\tau(j)}- c_{\tau(i)}^jv_{\tau(i)}-
c_{\tau(j)}^jv_{\tau(j)})z=0$. This implies that either
$c_{\tau(i)}^i=c_{\tau(i)}^j$ and $c_{\tau(j)}^i=c_{\tau(j)}^j$,
which is a contradiction since $\sigma(v_i)\neq \sigma(v_j)$, or
$z=0$. Thus $z=0$ and so $c_r^i=c_r^j=0$ for all $r\neq \tau(i),
\tau(j)$. Now we have

\noindent $(a_{\tau(ij)}^{ij}v_{\tau(i)\,\tau(j)}+b_1^{ij}v_1+\cdots
+b_m^{ij}v_m) (c_{\tau(i)}^i-c_{\tau(i)}^j)(v_{\tau(i)}-v_{\tau(j)})
\equiv$

\noindent $a_{\tau(ij)}^{ij}(c_{\tau(i)}^i-c_{\tau(i)}^j)
(v_{\tau(i)}^2-v_{\tau(j)}^2) +(b_1^{ij}v_1+\cdots
+b_m^{ij}v_m)(c_{\tau(i)}^i-c_{\tau(i)}^j) (v_{\tau(i)}-v_{\tau(j)})
=$

\noindent $((c_{\tau(i)}^{i})^2-(c_{\tau(i)}^j)^2)v_{\tau(i)}^2 +
(c_{\tau(i)}^ic_{\tau(j)}^i-c_{\tau(i)}^jc_{\tau(j)}^j)
(v_{\tau(i)}v_{\tau(j)}+v_{\tau(j)}v_{\tau(i)})+
((c_{\tau(j)}^{i})^2-(c_{\tau(j)}^j)^2)v_{\tau(j)}^2$

Because the right-hand side is in the subspace generated by
$v_{\tau(i)}, v_{\tau(j)}$, the left-hand side is as well.
Therefore, only $b_{\tau(i)}^{ij}, b_{\tau(j)}^{ij}$ can be nonzero.

Let us write down what we know so far.  For any $i,j, \, 1\leq
i,j\leq n$, we have:

\noindent 1)
$\sigma(v_{i\,j})=a_{\tau(ij)}^{ij}v_{\tau(i)\tau(j)}+b_{\tau(i)}^{ij}v_{\tau(i)}+
b_{\tau(j)}^{ij}v_{\tau(j)}$

\noindent 2)
$\sigma(v_i)=c_{\tau(i)}^iv_{\tau(i)}+c_{\tau(j)}^iv_{\tau(j)}$

\noindent 3)
$\sigma({v_j})=c_{\tau(i)}^jv_{\tau(i)}+c_{\tau(j)}^jv_{\tau(j)}$

and

\noindent 4)
$(a_{\tau(ij)}^{ij}v_{\tau(i)\tau(j)}+b_{\tau(i)}^{ij}v_{\tau(i)}+
b_{\tau(j)}^{ij}v_{\tau(j)})
(c_{\tau(i)}^i-c_{\tau(i)}^j)(v_{\tau(i)}-v_{\tau(j)}) \equiv$

\noindent $a_{\tau(ij)}^{ij}(c_{\tau(i)}^i-c_{\tau(i)}^j)
(v_{\tau(i)}^2-v_{\tau(j)}^2) +(b_{\tau(i)}^{ij}v_{\tau(i)}+
b_{\tau(j)}^iv_{\tau(j)})(c_{\tau(i)}^i-c_{\tau(i)}^j)
(v_{\tau(i)}-v_{\tau(j)}) =$

\noindent $((c_{\tau(i)}^{i})^2-(c_{\tau(i)}^j)^2)v_{\tau(i)}^2 +
(c_{\tau(i)}^ic_{\tau(j)}^i-c_{\tau(i)}^jc_{\tau(j)}^j)
(v_{\tau(i)}v_{\tau(j)}+v_{\tau(j)}v_{\tau(i)})+
((c_{\tau(j)}^{i})^2-(c_{\tau(j)}^j)^2)v_{\tau(j)}^2$

Applying (1) and (3) above to $\{ik\}$ (we are using here that
$|V_1|>2$) we find that $\sigma(v_i)= c_{\tau(i)}^iv_{\tau(i)} +
c_{\tau(k)}^iv_{\tau(k)}$. Because $\sigma(v_i)\neq 0$,
$c_{\tau(i)}^i \text{ and } c_{\tau(j)}^i$ cannot both be zero.
Furthermore, $\tau(i)$, $\tau(j)$, and $\tau(k)$ are distinct, so
$c_{\tau(i)}^i\neq 0 \text{ and }c_{\tau(j)}^i = c_{\tau(k)}^i =0$.
Thus, $\sigma(v_i)= c_{\tau(i)}^iv_{\tau(i)}$.

Because we have
$c_{\tau(i)}^i-c_{\tau(i)}^j=-(c_{\tau(j)}^i-c_{\tau(j)}^j)$ and
$c_{\tau(j)}^i = 0 = c_{\tau(i)}^j$, we see that $c_{\tau(i)}^i$ is
independent of $i$; call this coefficient c.  Thus,

\noindent $a_{\tau(ij)}^{ij}c (v_{\tau(i)}^2-v_{\tau(j)}^2)
+(b_{\tau(i)}^{ij}v_{\tau(i)}+ b_{\tau(j)}^{ij}v_{\tau(j)})c
(v_{\tau(i)}-v_{\tau(j)}) = c^2(v_{\tau(i)}^2-v_{\tau(j)}^2)$

\noindent $\Rightarrow b_{\tau(i)}^{ij}=b_{\tau(j)}^{ij}=0 \text{
and } a_{\tau(ij)}^{ij}=c$ for all $\{ij\}$.

What $\sigma$ does on level one forces what happens on the levels
above.   We may compose $\sigma$ with the automorphism that
multiplies each element by $1/c$; call this composition
$\hat{\sigma}$.  We have shown that $\hat{\sigma}$ permutes the
vertices in levels 1 and 2.  Assume that $\hat{\sigma}$ permutes the
vertices in levels less than or equal to $k-1$;i.e.
$\hat{\sigma}(v_B)=v_{\tau(B)}$.  Let $v_A,v_{A_1},...,v_{A_{m_k}}
\in V_k$, $v_B,v_C,v_{B_1},...,v_{B_{m_{k-1}}} \in V_{k-1}$.  Each
vertex $v_A$ in level $k$ is present in at least one relation
$v_{A}(v_{B}- v_{C})- v_{B}^2 +v_{C}^2 + (v_{B}-v_{C}) v_{B\cap
C}=0$.  Apply $\hat{\sigma}$ to this relation and we get
$(a_{A_1}^Av_{A_1}+\cdots +a_{A_{m_k}}^Av_{A_{m_k}}+
b_{B_1}^Av_{B_1}+\cdots
+b_{B_{m_k-1}}^Av_{B_{m_{k-1}}})(v_{\tau(B)}-v_{\tau(C)})-v_{\tau(B)}^2+
v_{\tau(C)}^2+(v_{\tau(B)}-v_{\tau(C)})v_{\tau(B\cap C)}$ by the
induction hypothesis.  In order for this to equal 0, we must have
that $v_A$ goes to $v_{\tau(B)\cup \tau(C)} = v_{\tau(B\cup C)} =
v_{\tau(A)}$.  Thus, $\hat{\sigma}(v_A)=v_{\tau(A)}$ for all $v_A\in
V$, and so $\tau\in$Aut$(\Gamma)$.

Therefore, Aut$(A(\Gamma))=k^*\times \text{Aut}(\Gamma)$.\end{proof}

\begin{lemma}\label{L:autdn}Aut$(\Gamma_{D_n}) = D_n$.\end{lemma}

\begin{proof}Any automorphism of the graph must preserve the set of
vertices at each level and so acts on the set $\{w_1,...,w_n\}$ of
all n vertices in level 1.  We may say $\sigma(w_i)=w_{\sigma(i)}$
(slightly abusing the use of $\sigma$).  Thus we can think of an
automorphism of the graph as being a permutation in $S_n$ acting on
the subscripts/labels of the vertices of level 1.  This will
uniquely determine what happens on higher levels; i.e.
$\sigma(v_{ij})=v_{\sigma(i)\,\sigma(j)}$. Labeling the vertices in
level two by the vertices they are connected to in level one ensures
that as long as the set of vertices in each level is preserved, the
edges will be as well.

Recall that $V_2$ refers to the vertices in level two of the graph.
 Only permutations which send
the set $V_2 = \{(i\, i+1):1\leq i\leq n\}$ to itself are allowed.
Clearly $r=(12...n)$ fixes $V_2$.  We may replace $\sigma$ by
$r^i\sigma$ for some $i$ and assume $\sigma(1)=1$.  Then
$\sigma(12)$ is either $(12)$ or $(1n)$, which implies either
$\sigma(2)=2$ (and thus $\sigma=id$) or $\sigma(2)=n$.  In the
latter case $\sigma=(2n)(3\, n-1)(4 n-2)\cdots =s$. Thus $r$ and $s$
generate the automorphism group of $\Gamma_{D_n}$; this is the
dihedral group on n elements, $D_n$.  Note that these automorphisms
may be viewed as reflections and rotations of the n-gon.\end{proof}

\begin{thm}\label{T:autdn} a) If $n \geq 3$,
Aut($A(\Gamma_{D_n}))=k^*\times D_n$, k the base field

b) If $n=2$, $$Aut(A(\Gamma_{D_2}))\cong \{M \in GL(3,k):M= \left[
\begin{array}{ccc}
c_1^1+c_2^1&c_1^2-c_2^1&c_2^1-c_1^2\\
0&c_1^1&c_2^1\\0&c_1^2&c_1^1+c_2^1-c_1^2\\ \end{array} \right],
c_i^j \in k \,\forall i,j\}$$
\end{thm}

\begin{proof}a) By Lemma \ref{L:autdn},
Aut$(\Gamma_{D_n}) = D_n$.  It is clear by looking at the graph
$\Gamma_{D_n}$ (Figure \ref{Fi:dn}) that for $n>2$ $\Gamma_{D_n}$
satisfies the conditions of Theorem \ref{T:autgen}. Therefore,
Aut($A(\Gamma_{D_n}))=k^*\times D_n$.

b)  In this case, $\Gamma_{D_2}$ fails to satisfy condition (i) of
Theorem \ref{T:autgen}; there are only two vertices on level 1.
Consider the proof of Theorem \ref{T:autgen}.  The proof is valid up
until we apply (1) and (3) to $\{ik\}$ to find that $\sigma(v_i)=
c_{\tau(i)}^iv_{\tau(i)} + c_{\tau(k)}^iv_{\tau(k)}$. In the case
where $n=2$, $\tau(i)+1 = \tau(i-1)+2 = \tau(i-1) \text{ in
}\mathbb{Z}/(2)$, so $c_{\tau(i)+1}^i = c_{\tau(i-1)}^i$ can be
nonzero. Thus, $w_i$ can go to a sum of multiples of $w_1$ and
$w_2$; $\sigma(w_i) = c_1^iw_1 + c_2^iw_2$. Because we only have one
vertex in level two, $v_{12}$, it can only go to a multiple of
itself plus multiples of $w_1$ and $w_2$.  Thus we can drop the sub
and superscripts on $a$ and the superscripts on $b_i$:
$\sigma(v_{12}) = av_{12}+b_1w_1 + b_2w_2$. We can rewrite (4) in
the proof of Theorem \ref{T:autgen} as $a(c_1^1-c_1^2)(w_1^2-w_2^2)
+ (b_1w_1+b_2w_2)(c_1^1-c_1^2)(w_1-w_2) = ((c_1^1)^2-(c_1^2)^2)
w_1^2 + (c_1^1c_2^1-c_1^2c_2^2)(w_1w_2+w_2w_1) +
((c_2^1)^2-(c_2^2)^2) w_2^2$.  We can conclude from this that
$a+b_1=c_1^1+c_1^2$, $a+b_2=c_2^1+c_2^2$, and
$-b_1(c_1^1-c_1^2)=c_1^1c_2^1-c_1^2c_2^2=b_2(c_1^1-c_1^2)
\Rightarrow -b_1=b_2$ (else $c_1^1c_2^1=c_1^2c_2^2 \Rightarrow
c_1^1=c_1^2 \text{ and } c_2^1=c_2^2$, which is not possible). These
imply that $2a=c_1^1+c_2^1+c_1^2+c_2^2 \Rightarrow a=c_1^1+c_2^1
\Rightarrow b_1=c_1^2-c_2^1=-b_2$.

Write the element $rv_{12}+sw_1+tw_2$ as the vector
$[\begin{array}{ccc}r & s & t\\ \end{array}]$.  Then a way to
visualize what this automorphism group looks like is to consider the
invertible transformation matrix M that sends $[\begin{array}{ccc}r
& s & t\\ \end{array}] \mapsto [\begin{array}{ccc}r & s & t\\
\end{array}]*M$ $$M= \left[
\begin{array}{ccc}
c_1^1+c_2^1&c_1^2-c_2^1&c_2^1-c_1^2\\
0&c_1^1&c_2^1\\0&c_1^2&c_1^1+c_2^1-c_1^2\\ \end{array} \right]$$
This matrix is conjugate to a triangular matrix and thus stabilizes
a flag.  The spaces $M$ stabilizes can be found by solving
$[\begin{array}{ccc}r & s & t\\ \end{array}]*M =
\alpha[\begin{array}{ccc}r & s & t\\ \end{array}]$.  $M$ stabilizes
the one-dimensional spaces $k\left[\begin{array}{ccc}1&-1&-1\\
\end{array} \right]$ and $k\left[\begin{array}{ccc}0&1&-1\\
\end{array} \right].$
\end{proof}

Denote the lattice of subsets of $\{1,...,n\}$ by
$\mathcal{L}_{[n]}$.  In other words, $Q_n = A(\mathcal{L}_{[n]})$.

\begin{lemma}\label{L:autqn}If $n\geq 3$, $Aut(\mathcal{L}_{[n]}) =
S_n$.\end{lemma}

\begin{proof}Any automorphism of the graph must preserve the set of vertices at
each level and so acts on the set $\{v_1,...,v_n\}$ of all n
vertices in level 1; so, we may say $\sigma(v_i)=v_{\sigma(i)}$
(slightly abusing the use of $\sigma$).  Thus we can think of an
automorphism of the graph as being a permutation in $S_n$ acting on
the subscripts/labels of the vertices of level 1.  This will
uniquely determine what happens on higher levels; i.e.
$\sigma(v_A)=v_{\sigma(A)}$.  Labeling the vertices in levels two
and higher by the vertices to which there is a path to in level one
ensures that as long as the set of vertices in each level is
preserved, the edges will be as well.  Since for each subset of
$\{1,...,n\}$ of cardinality $i$ level $i$ has a vertex labeled by
that subset, every element of $S_n$ is an automorphism of the graph.
In other words, for every $\tau \in S_n$, $\tau$ will permute the
vertices on each level.\end{proof}

\begin{thm}If $n\geq 3$, $Aut(Q_n)=k^* \times S_n$.\end{thm}

\begin{proof}We can see that $\mathcal{L}_{[n]}$ satisfies
conditions (i) and (ii) of Theorem \ref{T:autgen} since each subset
of $\{1,...,n\}$ occurs exactly once as a vertex in the lattice and
for $n>2$ there are more than 2 singleton subsets.  Condition (iii)
is satisfied because each vertex $v_B$ directly below a vertex $v_A$
is obtained by removing exactly one element from $A$.  Thus for any
vertex $v_C$ two levels below $v_A$, $|C|=|A|-2$.  Say
$C=A\backslash \{i,j\}$.  There are only two ways to obtain $C$:
first remove $i$ then $j$ or vice versa.  Therefore, there are only
two paths from $v_A$ to $v_C$.

Therefore, by Lemma \ref{L:autqn} and Theorem \ref{T:autgen},
Aut$(Q_n)=k^*\times S_n$.
\end{proof}

Consequently, in both of these algebras, the
Aut$(A(\Gamma))$-submodules of $A(\Gamma)_{[i]}$ are precisely the
Aut$(\Gamma)$-submodules.  Since Aut$(\Gamma)$ is finite we have
that $A(\Gamma)_{[i]}$ is a completely reducible
Aut$(A(\Gamma))$-module whenever characteristic $k=0$.

\section{Graded trace generating functions}\label{S:grtr}

Pass to the associated graded algebra, gr$A(\Gamma)$. Let $\phi_1,
..., \phi_l$ denote all of the distinct irreducible representations
of Aut$A(\Gamma)$ and let $\chi_j$ denote the character afforded by
$\phi_j$. Aut$A(\Gamma)$ acts on each $A(\Gamma)_{[i]}$, and so the
completely reducible Aut$A(\Gamma)$-module $A(\Gamma)_{[i]}$ may be
written as $\displaystyle \bigoplus_{j=1}^{l}m_{ij}\phi_j$.  The basis
$\mathcal{B}(\Gamma)$ of $A(\Gamma)$ is invariant under the
automorphism $\sigma$. Therefore, the trace of $\sigma$ on $gr
A(\Gamma)$, $Tr\sigma |_{A(\Gamma)}$, is the number of fixed basis
elements.

Remark: $Tr\sigma$ is the dimension of the subalgebra
$A(\Gamma^\sigma)$, which is not the same as the dimension of the
fixed point space.  The subalgebra $A(\Gamma^\sigma)$ described in
Section \ref{S:subalg} is the span of the set of fixed elements of
the basis.  On the other hand, the fixed point space is the span of
the sums of orbits of $\sigma$.  Averages over orbits are in the
fixed point space, but not in the subalgebra.

\bigskip

We will give two methods by which to find the graded trace
generating functions for general $A(\Gamma)$.  The first will be to
essentially count ``allowable'' and ``non-allowable'' words - a
generating function that gives the number of irreducible words in
each grading in the subalgebra $A(\Gamma^\sigma)$.  The second will
generalize Equation \ref{Eq:Hilbert} to use on the subgraph
$\Gamma^\sigma$ and subalgebra $A(\Gamma^\sigma)$.  These graded
trace generating functions will be used to find the multiplicities
of irreducible representations.

\subsection{Method 1 - Counting fixed words:}

The $Tr\sigma|_{A(\Gamma)_{[i]}}$ is the number of fixed basis
elements of degree i.  In other words, the number of sequences
$(x_1,k_1),...,(x_l,k_l)$ such that $1\leq k_j\leq|x_j|$,
$k_1+\cdots+k_l=i$, $e(x_i,k_1)\cdots e(x_l,k_l)$ is irreducible and
$\sigma x_j=x_j \,\forall j$.  Recall that $e(x_i,k)e(x_j,l)$ is
reducible if there is a path from $x_i$ to $x_j$ and the level of
$x_i$ equals the level of $x_j$ plus k.

\begin{lemma}\label{L:lattice}Let $X$ be a vector space with fixed
basis $\mathcal{B}$.  Let $Z = \{V\subseteq X: V \text{ subspace, }
V=\text{span}(V\cap \mathcal{B})\}$.  Then $Z(+,\cap)$ is a lattice
isomorphic to the lattice of $\mathcal{P(B)}(\cap,\cup)$.\end{lemma}

\begin{proof} The map $\phi:Z\rightarrow \mathcal{P(B)}
\text{ defined by } \phi(V) = V \cap \mathcal{B}$ is a lattice
isomorphism.
\end{proof}

\noindent Remark: The lattice of subsets of $\mathcal{B}$ is
distributive, and so $Z$ is distributive.

\begin{thm}\label{T:word} Let $X = \sum X_i$ be a graded vector
space and let the basis $\mathcal{X}$ of $X$ consist of homogeneous
elements; $\mathcal{X} = \cup X_i$, where $\mathcal{X}_i=
\mathcal{X}\cap X_i$. Then $T(X)$ is bi-graded with $T(X)_{i,j} =
\text{span}\{x_{l_1}\cdots x_{l_i} : x_{l_k} \in \mathcal{X}_{l_k},
l_1+\cdots+l_i = j\}$.  Let $\mathcal{Y}$ be a finite set of quadratic
monomials in $\mathcal{X}$, and define $Y = <\mathcal{Y}> \subseteq
T(X)$.  Let $|\cdot|$ denote the bi-graded dimension of the space.
Then
$$|T(X)/Y| = \frac{1}{1-|X|+|Y|-|XY\cap YX|+|X^2Y\cap XYX\cap Y
X^2|-\cdots}.$$\end{thm}

\begin{proof}
Denote $T(X)/Y$ by $A$ in this proof.  The graded dimension of
$T(X)$ is $\frac{1}{1-|X|}$. Because each generator in $Y$ is a
quadratic monomial, as vector spaces we can identify $A$ with the
subspace of $T(X)$ spanned by words not containing a subword in $Y$.

For $i\geq0$ define $Y^{(i)}:=\displaystyle \bigcap_{j=0}^{i}
X^{j}YX^{i-j}$; note that $Y^{(0)}=Y$.  It will also be convenient
to define $Y^{(-1)}:=X$ and $Y^{(-2)}:=k$.  For $i\geq-2$, let
$T_i:= T(X)Y^{(i)}/(T(X)YX^{i+2}T(X)\cap T(X)Y^{(i)})$.  Also,
define $T_{-3}:=T(X)/(T(X)XT(X))$.

Now define a map $\phi_i: T_i\rightarrow T_{i-1}$ for $i\geq-1$ by,
for $a\in T(X)Y^{(i)}$, \\$a+T(X)YX^{i+2}T(X)\cap T(X)Y^{(i)} \mapsto
a+T(X)YX^{i+1}T(X)\cap T(X)Y^{(i-1)}$.  Also, define $\phi_{-2}:
T_{-2}\rightarrow T_{-3}$ by $a+T(X)YT(X) \mapsto a+T(X)XT(X)$.
Because $Y^{(i)}=XY^{(i-1)}\cap YX^i \subseteq XY^{(i-1)}$,
$T(X)Y^{(i)}\subseteq T(X)Y^{(i-1)}$.  Also, $T(X)YX^{i+2}T(X)
\subseteq T(X)YX^{i+1}T(X)$ since
$X^{i+2}T(X)= \\X^{i+1}XT(X)\subseteq X^{i+1}T(X)$.  Thus, $\phi_i$ is
a well-defined map.

Next we will show that the sequence $\cdots \rightarrow T_j
\rightarrow T_{j-1}\rightarrow \cdots \rightarrow T_{-2} \rightarrow
T_{-3}\rightarrow 0$ is exact.  For $i\geq0$,
$\phi_{i-1}(\phi_i(a+T(X)YX^{i+2}T(X)\cap T(X)Y^{(i)})) = a +
T(X)YX^{i}T(X)\cap T(X)Y^{(i-2)}$.  Now $a\in T(X)Y^{(i)} \subseteq
T(X)Y^{(i-1)} \subseteq T(X)Y^{(i-2)}$.  Also, $a\in
T(X)Y^{(i)}=T(X)(YX^{i}\cap\cdots\cap X^{i}Y) \subseteq T(X)YX^i
\subseteq T(X)YX^iT(X)$. Thus, the image is $0$ and im$\phi_i
\subseteq$ ker$\phi_{i-1}$.

To show the other inclusion, note that ker$\phi_{i-1} = \{a+
T(X)YX^{i+1}T(X)\cap T(X)Y^{(i-1)}: \\a\in T(X)YX^{i}T(X) \cap
T(X)Y^{(i-2)}\}$.  For $a\in T(X)Y^{(i-1)}$ define $\bar{a} = a+
T(X)YX^{i+1}T(X)\cap T(X)Y^{(i-1)}$.  Assume that $\bar{a}\in
\text{ker}\phi_{i-1}$.  Then $a\in T(X)YX^iT(X)$. We want that $a\in
T(X)Y^{(i)}$.  To show this first observe that $a\in
T(X)Y^{(i-1)}\cap T(X)YX^iT(X) = T(X)Y^{(i-1)} \cap (T(X)YX^i +
T(X)YX^{i+1}T(X))$. Let $\mathcal{B}$ be the set of all monomials of
$\mathcal{X}$; $\mathcal{B}$ is a basis for $T(X)$ and
$\mathcal{Y}\subseteq \mathcal{B}$.  Consider the set of $X^iYX^j$.
Now $X^iYX^j$ is equal to the span of
$\mathcal{X}^i\mathcal{Y}\mathcal{X}^j$, and so Lemma
\ref{L:lattice} applies.  Therefore, the lattice generated by all
the $X^iYX^j$ is distributive.  Hence, we have that $T(X)Y^{(i-1)}
\cap (T(X)YX^i + T(X)YX^{i+1}T(X)) = T(X)Y^{(i-1)} \cap T(X)YX^i +
T(X)Y^{(i-1)} \cap T(X)YX^{i+1}T(X)$. Since $T(X)Y^{(i-1)} \cap
T(X)YX^i\subseteq T(X)Y^{(i)}$, $a\in T(X)Y^{(i)}$ and $\bar{a} \in
\text{im}\phi_i$. Therefore, im$\phi_i = \text{ker}\phi_{i-1}$ for
$i\geq 0$.

Since $\phi_{-1}$ is homogeneous in degree, im$\phi_{-1} = A\cap
T(X)X = \text{ker}\phi_{-2}$. Furthermore, the image of $\phi_{-2}$
is nonzero and maps to a one-dimensional space, and thus is
surjective.  Therefore, the sequence is exact.

Finally we will show that $T_i\cong AY^{(i)}$.  We will do this by
proving that $T(X)Y^{(i)} = AY^{(i)} \oplus T(X)YX^{i+2}T(X) \cap
T(X)Y^{(i)}$.  $T(X)Y^{(i)} = (A+T(X)YT(X))Y^{(i)} \subseteq
AY^{(i)} + T(X)YT(X)Y^{(i)} \subseteq AY^{(i)} + T(X)YX^{i+2}T(X)
\cap T(X)Y^{(i)}$.  Also, $AY^{(i)} \cap T(X)YX^{i+2}T(X)\subseteq
AX^{i+2} \cap T(X)YT(X)X^{i+2} \\\subseteq (A\cap T(X)YT(X))X^{i+2} =
(0)$.  Thus, our claim is proved and $T_i\cong AY^{(i)}$.  In
particular, $T_0= T(X)Y/(T(X)YX^2T(X)\cap T(X)Y) \cong AY$, \\$T_{-1}
= T(X)X/(T(X)YXT(X)\cap T(X)X)\cong AX$, \\$T_{-2} =
T(X)/T(X)YT(X)\cong A$, and \\$T_{-3} = T(X)/T(X)XT(X) \cong k$.
Hence, we can write our exact sequence as \\$\cdots \rightarrow
AY^{(j)}\rightarrow AY^{(j-1)} \rightarrow \cdots \rightarrow AY
\rightarrow AX \rightarrow A \rightarrow k \rightarrow 0$.

\noindent Therefore, $1=|k|= \sum(-1)^i|T_i| = |A| - |A||X| +\sum_{i\geq
0}(-1)^i|A||Y^{(i)}| = |A|(1-|X|+\sum_{i\geq 0}(-1)^i|Y^{(i)}|)$
implies that \[|A| = \frac{1}{1-|X|+\sum_{i\geq0}(-1)^i|Y^{(i)}|}\qedhere\]
\end{proof}

Define another grading on $A(\Gamma)$ by $A(\Gamma)_{[[k]]} =
\text{span}\{e(v_1,i_1)\cdots e(v_k,i_k): (v_j,i_j)\not\gtrdot
(v_{j+1},i_{j+1})\}$.  This induces an increasing filtration on
$A(\Gamma)$, $A(\Gamma)_{((k))} = \text{span}\{e(v_1,i_1)\cdots
e(v_j,i_j):j\leq k, \\(v_j,i_j)\not\gtrdot (v_{j+1},i_{j+1})\}$.
Thus as a vector space we can identify $A(\Gamma)$ with its
associated graded algebra, $gr' A(\Gamma)$.

\begin{lemma}\label{L:newgr}Define $W(\Gamma^\sigma) =
\text{span}\{e(v,k): l\geq 0, v\in V_\sigma, 1\leq k\leq |v|\}$ and
$R(\Gamma^\sigma)$ the two-sided ideal in $T(W(\Gamma^\sigma))$ generated by $\tilde{R}(\Gamma^\sigma) = \text{span}\{e(v,k)e(u,l): v>u \in
V_\sigma, k=|v|-|u|\}$. Then $A(\Gamma^\sigma) =
T(W(\Gamma^\sigma))/R(\Gamma^\sigma)$ as a subalgebra of
gr$'A(\Gamma)$.
\end{lemma}

\begin{proof}This follows from the definitions of $A(\Gamma^\sigma)$ and
of $gr'A(\Gamma)$.\end{proof}

\noindent $W(\Gamma^\sigma)$ and $R(\Gamma^\sigma)$ satisfy the
hypotheses for Theorem \ref{T:word}.  Therefore,
$$|A(\Gamma^\sigma)| = \frac{1}{1-|W(\Gamma^\sigma)|+|\tilde{R}(\Gamma^\sigma)|-|\tilde{R}(\Gamma^\sigma)W(\Gamma^\sigma)\cap
W(\Gamma^\sigma) \tilde{R}(\Gamma^\sigma)|+\cdots}.$$

\subsection{Method 2 - Generalizing the Hilbert Series Equation
(\ref{Eq:Hilbert}):}

We would like to apply the function $H(A(\Gamma),t)$ (see Equation
(\ref{Eq:Hilbert})) to the subalgebras created by our fixed points.
Take the subgraph of $\Gamma$ consisting of the points fixed by an
automorphism $\sigma$.  This generates a subalgebra of gr$A(\Gamma)$
in the way described in Section \ref{S:subalg}. Thus, we are using
Equation (\ref{Eq:Hilbert}) with the additional condition that the
vertices in the sum are fixed by $\sigma \in Aut(A(\Gamma))$; call
this modified formula $Tr_\sigma(A(\Gamma),t).$

\begin{thm}\label{T:trace}Let $\Gamma$ be a layered graph with
unique minimal element * of level 0 and $\sigma$ an automorphism of
the graph.  Let $\Gamma^\sigma$ be the subgraph of $\Gamma$ with
vertices being those fixed by $\sigma$ (as described in Section
\ref{S:subalg}).  Denote the Hilbert series of the subalgebra
$A(\Gamma^\sigma)$, which is the graded trace function of $\sigma$
acting on $A(\Gamma)$, by $Tr_\sigma(A(\Gamma),t)$ (or
$Tr_\sigma(t)$ when $A(\Gamma)$ is clear).  Then
\begin{equation}\label{Eq:trace}
Tr_\sigma(A(\Gamma),t)=\frac{1-t}{1-t\displaystyle
\sum_{\substack{v_1>\cdots>v_l\geq
*\\v_1,...,v_l \in V_\sigma}}(-1)^{l-1} t^{|v_1|-|v_l|}}.\end{equation}
\end{thm}

\begin{proof}
Write $Tr(t)$ for $Tr_\sigma(A(\Gamma),t)$ in this proof.  Let
$v_1,...,v_l, v, w\in V_\sigma$.  Recall that the basis for
$A(\Gamma^\sigma)$ is $\mathcal{B}_\sigma = \{e(v_1,k_1)\cdots e(v_l,k_l): v_1,...v_l\in V_\sigma, 1\leq k_i\leq |v_i|, e(v_i,k_i)\mbox{$\not\gtrdot$}
e(v_{i+1},k_{i+1})\}$.

For $v\in (V_\sigma)_+$, define $C_v =
\displaystyle\bigcup_{k=1}^{|v|} e(v,k)\mathcal{B}_\sigma, \, B_v =
C_v\cap \mathcal{B}_\sigma, \, D_v = C_v\backslash B_v$.  Then
$\mathcal{B}_\sigma = \{*\}\cup \displaystyle \bigcup_{v\in
(V_\sigma)_+}B_v$.  Let $Tr_v = Tr_\sigma(B_v,t)$, the graded
dimension of the span of $B_v$.  Then $Tr(t) = 1+\displaystyle
\sum_{v\in (V_\sigma)_+}Tr_v(t)$.  We also have $Tr_\sigma(C_v,t) =
(t+\cdots+t^{|v|})Tr(t) = t(\frac{t^{|v|}-1}{t-1})Tr(t)$ and,
because $D_v = \displaystyle \bigcup_{v>w>*}e(v,|v|-|w|)B_w$,
$Tr_\sigma(D_v,t) = \displaystyle \sum_{v>w>*}t^{|v|-|w|}Tr_w(t)$.
Thus,
\begin{equation}\label{Eq:trv}Tr_v(t) =
t\left(\frac{t^{|v|}-1}{t-1}\right)Tr(t) - \displaystyle
\sum_{v>w>*}t^{|v|-|w|}Tr_w(t).\end{equation}

This equation may be written in matrix form.  Put an order on
$V_\sigma$, arrange the elements in decreasing order, and index the
elements of vectors and matrices by this ordered set.  Let
$\vec{Tr}(t)$ denote the column vector with $Tr_v(t)$ in the
$v$-position and 0 in the *-position.  Let $\vec{s}$ denote the
column vector with entry $t^{|v|}-1$ in the $v$-position, let
$\vec{1}$ denote the column vector having 1 as each entry, and let
$\zeta(t)$ denote the matrix with the entry in the $(v,w)$-position
being $t^{|v|-|w|}$ if $v\geq w$ and 0 otherwise. Then rewriting
Equation \ref{Eq:trv} gives $\zeta(t)\vec{Tr}(t) =
\frac{t}{t-1}\vec{s}Tr(t)$.

Now $\zeta(t)-I$ is a strictly upper triangular matrix, and so
$\zeta(t)$ is invertible; $\zeta^{-1}(t) =
I-(\zeta(t)-I)+(\zeta(t)-I)^2-\cdots$.  Thus the $(v,w)$-entry of
$\zeta^{-1}(t)$ is $\displaystyle
\sum_{\substack{v=v_1>\cdots>v_l=w\geq
*\\v_1,...,v_l \in V_\sigma}}(-1)^{l+1}t^{|v_1|-|v_l|}$.

Thus we can multiply $\zeta(t)\vec{Tr}(t) =
\frac{t}{t-1}\vec{s}Tr(t)$ by $\zeta^{-1}$ and then multiply by
$\vec{1}^T$ to obtain $\vec{1}^T\vec{Tr} = Tr(t)-1 =
\frac{t}{t-1}\vec{1}^T\zeta^{-1}(t)\vec{s}Tr(t)$. Solving for
$Tr(t)$ we get
\[Tr(t) = \frac{t-1}{t-1-t\vec{1}^T\zeta^{-1}(t)\vec{s}} = \frac{1-t}{1-t\vec{1}^T\zeta^{-1}(t)\vec{1}} =
\frac{1-t}{1-t\displaystyle\sum_{\substack{v_1>\cdots>v_l\geq
*\\v_1,...,v_l \in V_\sigma}}(-1)^{l-1}t^{|v_1|-|v_l|}}.\]
\end{proof}

\noindent Remark 1: $\zeta$ is the standard zeta matrix and
$\zeta^{-1}$ is the M$\ddot{o}$bius matrix when $t=1$.

\noindent Remark 2: We will normally apply this method; although,
both methods can theoretically be applied in all layered graph
algebras.

In general, if $\sigma$ and $\tau$ are conjugate, $\Gamma^\sigma$
and $\Gamma^\tau$ are isomorphic (by the conjugation acting on the
subscripts of vertices). Thus, it is enough to find the graded trace
functions for any one representative of each conjugacy class.

\section{Graded Trace Generating Functions for $A(\Gamma_{D_n})$ and
$Q_n$}\label{S:genfn}

Let us calculate the generating functions for our algebras.  First
of all, the graded trace of the identity acting on the algebra is
the graded dimension of the algebra.  We will derive it using the
theorems above to show that we get the same result as the Hilbert
series given earlier.

\bigskip

\subsection{The Algebra $A(\Gamma_{D_n})$}

\medskip

We will now find $Tr_\sigma(A(\Gamma_{D_n}),t) =
\frac{1}{1-(a_1t+a_2t^2+a_3t^3)}$ using Method 1 in Section
\ref{S:grtr}. Notice that because $\Gamma_{D_n}$ has three levels,
$\tilde{R}(\Gamma_{D_n}^{id})^{(i)} = \{0\}$ for $i\geq 2$.
$W(\Gamma_{D_n}^{id})$ has basis \\$\{e(u,3), e(u,2), e(u,1),
e(v_{i\,i+1},2), e(v_{i\,i+1},1), e(w_i,1), \, 1\leq i \leq n\}$.
Hence $|W(\Gamma_{D_n}^{id})| = (2n+1)t+ (n+1)t^2 + t^3$. The
reducible words of degree 2 are $e(u,2)e(w_i,1)$,
$e(u,1)e(v_{i\,i+1},2)$, $e(u,1)e(v_{i\,i+1},1)$,
$e(v_{i\,i+1},1)e(w_i,1)$, $e(v_{i\,i+1},1)e(w_{i+1},1)$, $1\leq
i\leq n$. The set of these words is a basis for
$\tilde{R}(\Gamma_{D_n}^{id})$, so we have
$|\tilde{R}(\Gamma_{D_n}^{id})| = 3nt^2+2nt^3$. The overlaps of
reducible words are $\{e(u,1)e(v_{i\,i+1},1)e(w_i,1),
\\e(u,1)e(v_{i\,i+1},1)e(w_{i+1}):1\leq i\leq n\}$. This set is a
basis for $\tilde{R}(\Gamma_{D_n}^{id})W(\Gamma_{D_n}^{id})\cap
W(\Gamma_{D_n}^{id}) \tilde{R}(\Gamma_{D_n}^{id})$ and so
$|\tilde{R}(\Gamma_{D_n}^{id})W(\Gamma_{D_n}^{id})\cap
W(\Gamma_{D_n}^{id}) \tilde{R}(\Gamma_{D_n}^{id})| = 2nt^3$.

\noindent Hence $a_3=1-(n+n)+(n+n)=1$, $a_2=(1+n)-(n+n+n)=1-2n$, and
$a_1=1+n+n=2n+1$.

\noindent Thus, we have $$Tr_{id}(A(\Gamma_{D_n}),t) =
\frac{1}{1-((2n+1)t-(2n-1)t^2+t^3)},$$ which agrees with the earlier
result.

\smallskip

Now since only u is fixed by $r^i$, $W(\Gamma_{D_n}^{r^i}) =
\text{span}\{e(u,3), e(u,2), e(u,1)\} \text{ and } \tilde{R}(\Gamma_{D_n}^{r^i})
= \{0\}$. Hence,
$$Tr_{r^i}(A(\Gamma_{D_n}),t) = \frac{1}{1-(t+t^2+t^3)}.$$

\noindent If $n$ is even, $W(\Gamma_{D_n}^s)$ has basis $\{e(u,3),
e(u,2), e(u,1), e(v_{12},2), e(v_{12},1), e(v_{n/2+1\,n/2+2},2),
\\e(v_{n/2+1\,n/2+2},1)\}$ and $\tilde{R}(\Gamma_{D_n}^s)$ has basis
$\{e(u,1)e(v_{i\,i+1},2), \, e(u,1)e(v_{i\,i+1},1):  i=1,n/2+1\}$.
Also, $W(\Gamma_{D_n}^{rs})$ has basis $\{e(u,3), e(u,2), e(u,1),
e(w_2,1), e(w_{n/2+2},1)\}$  and $\tilde{R}(\Gamma_{D_n}^{rs})$ has basis
\\$\{e(u,2)e(w_2,1), \, e(u,2)e(w_{n/2+2},1)\}$.  If $n$ is odd,
$W(\Gamma_{D_n}^{s})$ has basis \\$\{e(u,3), e(u,2), e(u,1),
e(v_{12},2), e(v_{12},1), e(w_{(n+3)/2},1)\}$ and
$\tilde{R}(\Gamma_{D_n}^{s})$ has basis \\$\{e(u,1)e(v_{12},2), \,
e(u,1)e(v_{12},2),\, e(u,2)e(w_{(n+3)/2},1)\}$.  Computing each separately we see that
$$Tr_{s}(A(\Gamma_{D_n}),t) = Tr_{rs}(A(\Gamma_{D_n}),t) = \frac{1}{1-(3t+t^2-t^3)}.$$

Using Method 2 we can get our graded trace generating functions by
applying Equation (\ref{Eq:trace}) to the subalgebras of
$A(\Gamma_{D_n})$. The automorphism $r^i$ only fixes u and the
minimal vertex (see Figure \ref{Fi:dnr}). Since we have two
vertices, no vertices one or two levels apart, and one pair of
vertices three levels apart (and only one path between them),
$$Tr_{r^i}=\frac{1-t}{1-(2t-t^4)}=\frac{1-t}{1-t(2-t^3)}=\frac{1}{1-(t+t^2+t^3)}.$$

\begin{figure}[h]
\begin{center}
\includegraphics[height=.75in]{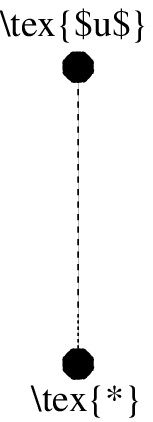}
\end{center}
\caption{$\Gamma_{D_n}^r$} \label{Fi:dnr}
\end{figure}

The automorphism s acting on the algebra when n is even fixes the
top vertex, the minimal vertex, and two vertices on level two
($v_{12}$ and $v_{n/2+1\,n/2+2}$)(see Figure \ref{Fi:dns}).
Similarly, when n is odd s fixes the top vertex, the minimal vertex,
and one vertex on each of levels one and two ($v_{12}$ and
$w_{(n+3)/2}$). Finally, rs fixes the top vertex, the minimal
vertex, and two vertices on level one ($w_2$ and $w_{n/2+2}$). Thus,
in each case, there are 4 vertices, two edges of length one, and two
of length two.  For the coefficient of $t^4$, we have $u>*$,
$u>\text{vertex}>*$, and $u>\text{vertex}>*$. Thus,
$$Tr_s(t) = Tr_{rs}(t) =
\frac{1-t}{1-(4t-2t^2-2t^3+t^4)}=\frac{1-t}{1-t(2-t)(2-t^2)}=\frac{1}{1-(3t+t^2-t^3)}.$$

\begin{figure}[h]
\begin{center}
\includegraphics[height=1in]{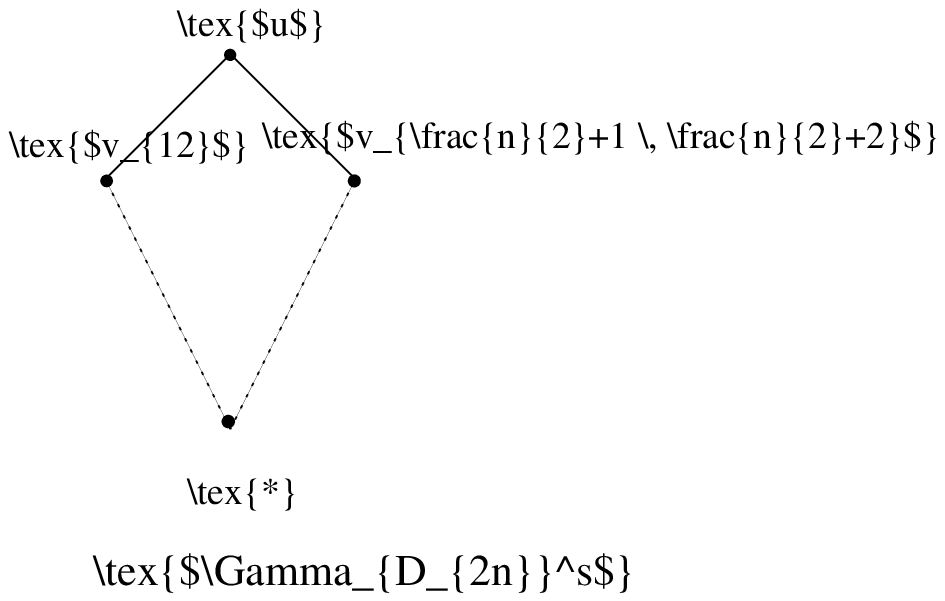}
\hfill
\includegraphics[height=1in]{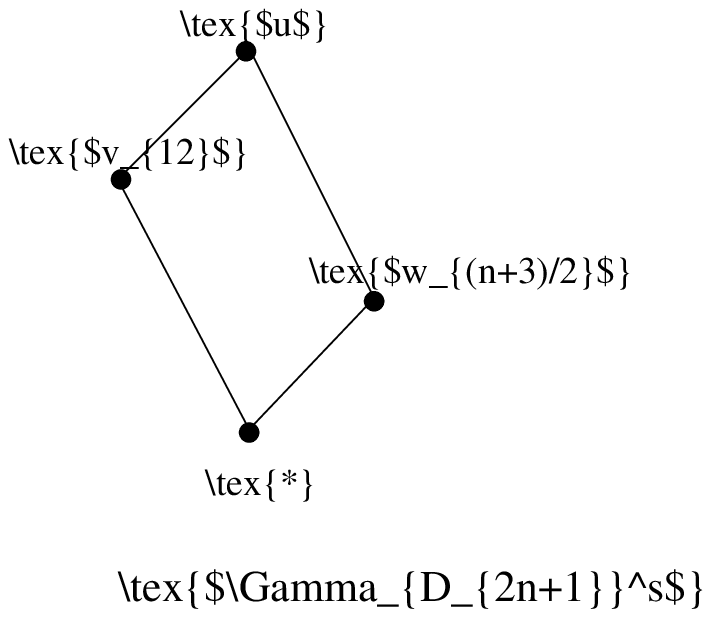}
\hfill
\includegraphics[height=1in]{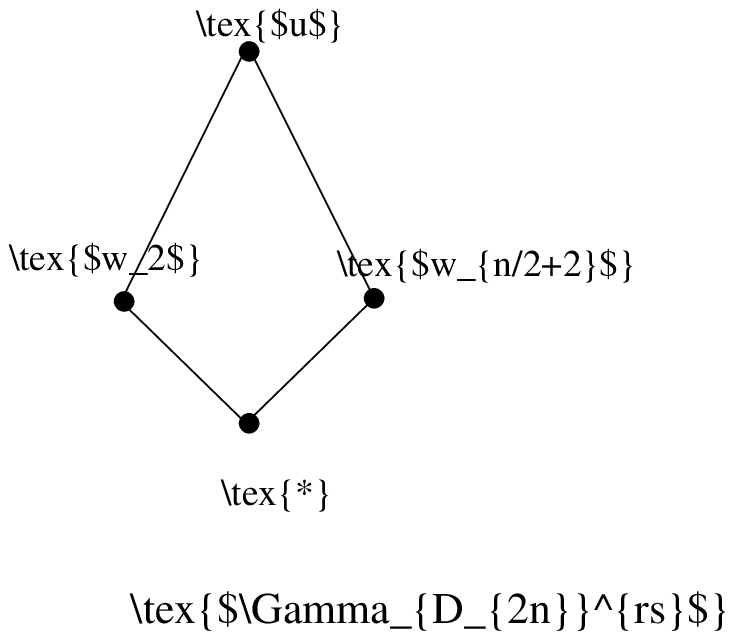}
\end{center}
\caption{$\Gamma_{D_n}^\sigma$} \label{Fi:dns}
\end{figure}

\bigskip

\subsection{The Algebra $Q_n$}

\medskip

\noindent Recall that $V_\sigma$ denotes the set of vertices in $\Gamma$ fixed
by $\sigma$.

\begin{thm}\label{T:trqn}Let $\sigma \in S_n$ and $\sigma = \sigma_1 \cdots \sigma_m$ be its cycle
decomposition.  Denote the length of $\sigma_j$ by $i_j$ (so $i_j
\geq 1, \, i_1+\cdots +i_m=n$).  Then
\begin{equation}\label{Eq:qntr}Tr_\sigma(Q_n,t)=\frac{1-t}{1-t\displaystyle
\prod_{j=1}^m (2-t^{i_j})}.\end{equation}\end{thm}

\noindent First of all, notice that when $\sigma=(1)$, this yields $H(Q_n,t)$
given above.

We will prove this using Equation \ref{Eq:trace} from Section
\ref{S:grtr}:
\[Tr_\sigma(A(\Gamma),t)=\frac{1-t}{1-t\displaystyle \sum_{\substack{v_1>\cdots>v_l\geq
*\\v_1,...,v_l \in V_\sigma}}(-1)^{l-1} t^{|v_1|-|v_l|}}.\]

\noindent If $w \subseteq \{1,...,n\}$ is $\sigma$-invariant, let
$\|w\|$ be the number of $\sigma$-orbits in w. Also, let
$\mathcal{O}_j$ denote the non-trivial orbit of $\sigma_j$.

The following Lemma and Corollary and their proofs are parallel to
[\cite{RSW},Lemma 2] and

\noindent [\cite{RSW},Corollary 1].  In the case
where $\sigma$ is the identity, they are the same.

\begin{lemma}Let $w \subseteq \{1,...,n\}$ be fixed by $\sigma$.
Then
$\displaystyle \sum_{w=w_1\supset ...\supset w_l=\emptyset} (-1)^l =
(-1)^{\|w\|+1}$ where the sum is over all chains of $\sigma$-fixed subsets of $w$.\end{lemma}

\begin{proof}
If $\|w\|=1$, then $w=\mathcal{O}_j$ for some j.  Thus, there are no
fixed proper subsets of w, and we get that both sides are equal to
1.  Assume the result holds for all sets with $\|\cdot \| < \|w\|$.
Then $\displaystyle \sum_{w\supset \cdots\supset w_l=\emptyset}
(-1)^l = \displaystyle \sum_{w\supset w_2 \supseteq \emptyset}
\displaystyle \sum_{w_2 \supset \cdots\supset w_l=\emptyset} (-1)^l
= \displaystyle \sum_{w\supset w_2 \supseteq \emptyset}
(-1)^{\|w_2\|}$ by the induction assumption. \\Now $\displaystyle
\sum_{w\supset w_2 \supseteq \emptyset} (-1)^{\|w_2\|} =
\displaystyle \left(\sum_{w\supseteq w_2 \supseteq \emptyset}
(-1)^{\|w_2\|}\right)-(-1)^{\|w\|}$.  Say $w=\mathcal{O}_1\cup
\mathcal{O}_2\cup \cdots \cup \mathcal{O}_r$.  Then the number of
$w_2$ such that $w_2\subseteq w$ and $\|w_2\|=i$ is $\binom{r}{i}$.
Hence, we have $\displaystyle \sum_{w\supset w_2 \supseteq
\emptyset} (-1)^{\|w_2\|} = \displaystyle \left(\sum_{w\supseteq w_2
\supseteq \emptyset} (-1)^{\|w_2\|}\right)-(-1)^{\|w\|} =
\displaystyle \sum_{i=0}^r \binom{r}{i} (-1)^i + (-1)^{\|w\|+1} =
(-1)^{\|w\|+1}$, as desired, since the alternating sum of the
binomial coefficients is zero.
\end{proof}

\begin{coro}\label{C:coro}Let $\{1,...,n\}\supseteq v\supseteq w$ be fixed by $\sigma$.  Then

$\displaystyle \sum_{v=v_1\supset v_2\supset \cdots \supset
v_l=w}(-1)^l = (-1)^{\|v\|-\|w\|+1}$ where the sum is over all
chains of $\sigma$-fixed subsets.\end{coro}

\begin{proof}
Let $w'$ denote the complement of $w$ in $v$.  Sets $u$ invariant
under $\sigma$ and satisfying $v\supseteq u\supseteq w$ are in
one-to-one correspondence with $\sigma$-invariant subsets of $w'$
via the map $u\mapsto u \cap w'$.  Thus, $\displaystyle
\sum_{v=v_1\supset v_2\supset \cdots\supset v_l=w}(-1)^l =
\displaystyle \sum_{w'=v'_l\supset \cdots\supset
v'_1=\emptyset}(-1)^l = (-1)^{\|w'\|+1}$ by the lemma. Because
$\|w'\|=\|v\|-\|w\|$, this gives us what we want.\end{proof}

\begin{proof}[Proof of Theorem \ref{T:trqn}]

By Corollary \ref{C:coro}, \begin{align*}&\sum_{v_1\supset\cdots
\supset v_l\supseteq\emptyset} (-1)^lt^{|v_1|-|v_l|+1} =
\sum_{\{1,...,n\}\supseteq v_1\supseteq v_l\supseteq
\emptyset} (-1)^{\|v_1\|-\|v_l\|+1}t^{|v_1|-|v_l|+1}\\
&=
\sum_{\substack{w,v_l\\w\cap v_l=\emptyset}} (-1)^{\|w\|+1}t^{|w|+1}
= \sum_{w}2^{m-\|w\|} (-1)^{\|w\|+1}t^{|w|+1}\end{align*}

\noindent The $\sigma$-invariant sets $w$ are unions of
$\sigma$-orbits. Write $a_j=1$ if $\mathcal{O}_j$ is contained in
$w$ and $a_j=0$ if not.  Then the m-tuple $\{a_1,...,a_m\}$ tells us
which orbits are contained in $w$.  We can then write $\displaystyle
\sum_{w}2^{m-\|w\|} (-1)^{\|w\|+1}t^{|w|+1}$ as $\displaystyle
\sum_{a_1,...,a_m\in \{0,1\}} (-1)^{\sum a_j + 1} 2^{m-\sum a_j}
t^{\sum (a_ji_j)+1}$. This equals \begin{align*}&
-t\sum_{a_1,...,a_m \in\{0,1\}} \prod_{j=1}^m
(-1)^{a_j}2^{1-a_j}t^{a_ji_j}\\&= -t \prod_{j=1}^m \sum_{a_j=0}^1
(-1)^{a_j}2^{1-a_j}t^{a_ji_j}\\&= -t\prod_{j=1}^m (2-t^{i_j})
\end{align*}

\noindent Therefore, we have (\ref{Eq:qntr}).  \end{proof}

\begin{eg}Here are the graded trace functions for $Q_4$:
\[Tr_{(1)}(Q_4,t) = \frac{1-t}{1-t(2-t)^4} \qquad Tr_{(12)}(Q_4,t) =
\frac{1-t}{1-t(2-t^2)(2-t)^2}\]
\[Tr_{(123)}(Q_4,t) =
\frac{1-t}{1-t(2-t^3)(2-t)} \qquad Tr_{(12)(34)}(Q_4,t) =
\frac{1-t}{1-t(2-t^2)^2} \]
\[Tr_{(1234)}(Q_4,t) =
\frac{1-t}{1-t(2-t^4)}\]
\end{eg}

\section{Representations of Aut$(A(\Gamma))$ acting on
$A(\Gamma)$}\label{S:reps}

Now let us determine the multiplicities of the irreducible
representations.  Assume $A(\Gamma)_{[i]}$ is a completely reducible
Aut$(\Gamma)$-module.  Note that this is true in our examples.  Fix
n. Let the graded trace generating function be denoted by
$Tr_{\sigma}(t)=\sum_i Tr_{\sigma,i}t^i$ where $Tr_{\sigma,i}=Tr
\sigma |_{A(\Gamma)_{[i]}}$.  Let $\phi$ be an irreducible
representation of Aut$(\Gamma)$ and $m_{\phi}(t)=\sum_i
m_{\phi,i}t^i$ where $m_{\phi,i}$ is the multiplicity of $\phi$ in
$A(\Gamma)_{[i]}$.  Finally, let the matrix $C=[\chi_{\sigma \phi}]$
where $\chi_{\sigma\phi}$ is the trace of $\sigma$ on the module
which affords the irreducible representation $\phi$; i.e. $C$ is the
character table of Aut$(\Gamma)$.

Then, if we fix the degree, $Tr_{\sigma,i}=\sum_{\phi} \chi_{\sigma
\phi}m_{\phi, i}$; so we have $Tr_{\sigma}(t)=\sum_{\phi}
\chi_{\sigma \phi}m_{\phi}(t)$.  Write
$\vec{Tr}(t)=[Tr_{\sigma_1}(t) ... Tr_{\sigma_l}(t)]^T$ and
$\vec{m}(t)=[m_{\phi_1}(t) ... m_{\phi_l}(t)]^T$.  Finally,
$$\vec{Tr}(t)=C^T\vec{m}(t) \implies \vec{m}(t)=(C^{T})^{-1}\vec{Tr}(t).$$

\subsection{Representations of Aut$(A(\Gamma_{D_n}))$ acting
on $A(\Gamma_{D_n})$}

Recall that the character table for $D_{n}$ where $n=2m$ is even is:

\begin{center}

\begin{tabular}{r|cccccccc}
&1&r&...&$r^j$&...&$r^m$&s&rs\\
\hline $\chi_{triv}$&1&1&...&1&...&1&1&1\\
$\chi_{1-1}$&1&1&...&1&...&1&-1&-1\\
$\chi_{-11}$&1&-1&...&$(-1)^j$&...&$(-1)^m$&1&-1\\
$\chi_{-1-1}$&1&-1&...&$(-1)^j$&...&$(-1)^m$&-1&1\\
$\chi_k$&2&$2\cos(2\pi k/n)$&...&$2\cos(2\pi kj/n)$&...&$2\cos(2\pi
km/n)$&0&0\\
\end{tabular}

\end{center}

\noindent where ($1\leq k\leq m-1$), $r=(12...n)$, and $s=(12)(3n)(4\,
n-1)...(\frac{n}{2}+1\,\, \frac{n}{2}+2)$;
so, \\$rs=(13)(4n) ... (\frac{n}{2}+1\,\, \frac{n}{2}+3)$.

When $n=2m+1$ is odd the character table is:

\begin{center}

\begin{tabular}{r|ccccccc}
&1&r&...&$r^j$&...&$r^m$&s\\
\hline $\chi_{triv}$&1&1&...&1&...&1&1\\
$\chi_{1-1}$&1&1&...&1&...&1&-1\\
$\chi_k$&2&$2\cos(2\pi k/n)$&...&$2\cos(2\pi kj/n)$&...&$2\cos(2\pi
km/n)$&0\\
\end{tabular}

\end{center}

\noindent where ($1\leq k\leq m$), $r=(12...n)$ and
$s=(12)(3n)...(\frac{n+1}{2}\,\, \frac{n+5}{2})$.

\begin{prop}Let $\vec{m}(t)$ be the vector of the graded multiplicities of the irreducible representations
of $D_n$ as described above.  Set
\\$a=\frac{1}{1-((2n+1)t-(2n-1)t^2+t^3)}$,
$b=\frac{1}{1-(t+t^2+t^3)}$, and $c=\frac{1}{1-(3t+t^2-t^3)}$.

a) Let n be even.  Then,
$$\vec{m}(t)=\left[ \begin{array}{c}
\frac{1}{2n}a+\frac{n-1}{2n}b+\frac{1}{2}c \\
\frac{1}{2n}a+\frac{n-1}{2n}b-\frac{1}{2}c \\
\frac{1}{2n}(a-b) \\
\frac{1}{2n}(a-b) \\
\frac{1}{n}(a-b)\\
 \vdots \\
 \frac{1}{n}(a-b)\\ \end{array} \right]$$

b) Let n be odd.  Then,
$$\vec{m}(t)=\left[ \begin{array}{c}
\frac{1}{2n}a+\frac{n-1}{2n}b+\frac{1}{2}c \\
\frac{1}{2n}a+\frac{n-1}{2n}b-\frac{1}{2}c \\
\frac{1}{n}(a-b)\\
 \vdots \\
 \frac{1}{n}(a-b)\\ \end{array} \right]$$

This is obtained from deleting the third and fourth entries in the n
is even case.
\end{prop}

\begin{proof} We verify this claim by multiplying the transpose of the character table of $D_n$
by $\vec{m}(t)$.  The result is
$$\vec{Tr}(t)=\left[
\begin{array}{c}
a \\
b \\
\vdots \\
b \\
c \\
c \\ \end{array} \right]$$ as desired.\end{proof}

Notice that all of the representations are realized; and, with large
multiplicity.

\subsection{Representations of $S_n$ acting on $Q_n$}

Unlike the $A(\Gamma_{D_n})$ case, we cannot write down one table
giving all of the values in terms of the graded trace functions.
However, we can give them in terms of the Frobenius formula.  First,
however, we will give an example.

\begin{eg}\label{E:repsq4}Irreducible Representations for $Q_4$:

The character table for $S_4$ is:

\begin{center}
\begin{tabular}{l|ccccc}
&(1)&(12)&(123)&(1234)&(12)(34)\\
\hline
$\chi_{triv}$&1&1&1&1&1\\
$\chi_{sgn}$&1&-1&1&-1&1\\
$\chi_3$&2&0&-1&0&2\\
$\chi_{reg}$&3&1&0&-1&-1\\
$\chi_{sgn\otimes reg}$&3&-1&0&1&-1
\end{tabular}
\end{center}

\medskip

Let $a=Tr_{(1)}(Q_4,t) = \frac{1-t}{1-t(2-t)^4} \qquad b=
Tr_{(12)}(Q_4,t) = \frac{1-t}{1-t(2-t^2)(2-t)^2}$,

$c = Tr_{(123)}(Q_4,t) = \frac{1-t}{1-t(2-t^3)(2-t)} \qquad d =
Tr_{(1234)}(Q_4,t) = \frac{1-t}{1-t(2-t^4)}$

$e= Tr_{(12)(34)}(Q_4,t) = \frac{1-t}{1-t(2-t^2)^2}$
\smallskip

Then, the multiplicities of the irreducible representations of $S_4$
acting on $Q_4$ as sums of the graded trace generating functions
are:

$m_{triv} =
\frac{1}{24}a+\frac{1}{4}b+\frac{1}{3}c+\frac{1}{4}d+\frac{1}{8}e$

$m_{sgn} =
\frac{1}{24}a-\frac{1}{4}b+\frac{1}{3}c-\frac{1}{4}d+\frac{1}{8}e$

$m_{3} = \frac{1}{12}a-\frac{1}{3}c+\frac{1}{4}e$

$m_{reg} = \frac{1}{8}a+\frac{1}{4}b-\frac{1}{4}d-\frac{1}{8}e$

$m_{sgn\otimes reg} = \frac{1}{8}a-\frac{1}{4}b
+\frac{1}{4}d-\frac{1}{8}e$

The numerical values for the first few degrees are given below:

\begin{center}
\begin{tabular}{l|ccccc}
&$\chi_{triv}$&$\chi_{sgn}$&$\chi_{3}$&$\chi_{reg}$&$\chi_{sgn\times reg}$\\
\hline $m_{\phi,1}$&4&0&1&3&0\\
$m_{\phi,2}$&26&1&17&36&13\\
$m_{\phi,3}$&219&54&239&434&273
\end{tabular}
\end{center}
\end{eg}

We can also write the multiplicities in terms of Frobenius' formula.
First we will give the notation used in the formula.   Let $C_i$ be
a representative from the conjugacy class $i$ and $i_j$ be the
number of j-cycles in $i$.  Also, let $\lambda$ be a partition of n
(representing an irreducible representation),
$\Delta(x)=\displaystyle \prod_{i<j}(x_i-x_j)$ and $P_j(x) =
x_1^j+\cdots+x_k^j$ where $k$ is at least the number of rows in
$\lambda$.  Set $l_1=\lambda_1+k-1,\,
l_2=\lambda_2+k-2,...,l_k=\lambda_k$.  Finally, for $f(x)\in
\mathbb{C}[x_1,...,x_k]$ the scalar $f(x)_{l_1,...,l_k}$ is defined
by $f(x) = \displaystyle \sum_{l_1,...,l_k} f(x)_{l_1,...,l_k}
x_1^{l_1}x_2^{l_2}\cdots x_k^{l_k}$. Then Frobenius' formula says
$\chi_\lambda(C_i) = [\Delta(x) \prod_j
P_j(x)^{i_j}]_{l_1,...,l_k}$.

\begin{prop}\label{P:frob}Let $r$ denote the degree and $\lambda_l$
the irreducible representation.  Then \[m_{\lambda_l,r} = \left[
\frac{1}{n!} \displaystyle \sum_{j = \text{partition of
n}}\chi_{\lambda_l}(C_j) |\mathcal{C}(j)| Tr_{(j)}
\right]_{(l_1,...,l_k,r)}\]\end{prop}

\begin{proof}Let \[S = \left[ \begin{array}{ccc} \chi_{\lambda_1}(C_1)&\cdots&\chi_{\lambda_1}(C_k)\\  \vdots& &\vdots\\
\chi_{\lambda_k}(C_1)&\cdots&\chi_{\lambda_k}(C_k) \end{array}\right]\] be the character table of $S_n$.  By the orthogonality relations, \[S^TS=D= \left[ \begin{array}{cccc} \sum_i \chi_{\lambda_i}(C_1)^2&0&\cdots&0\\  0&\ddots& &\vdots\\ \vdots&&\ddots&0\\
0&\cdots&0&\sum_i \chi_{\lambda_i}(C_k)^2 \end{array}\right].\]  Thus \[S^{-1}=D^{-1}S^T = \left[ \begin{array}{ccc} \chi_{\lambda_1}(C_1)/\sum_i \chi_{\lambda_i}(C_1)^2&\cdots&\chi_{\lambda_k}(C_1)/\sum_i \chi_{\lambda_i}(C_1)^2\\  \vdots& &\vdots\\
\chi_{\lambda_1}(C_k)/\sum_i
\chi_{\lambda_i}(C_k)^2&\cdots&\chi_{\lambda_k}(C_k)/\sum_i
\chi_{\lambda_i}(C_k)^2 \end{array}\right].\]

\noindent Because $\vec{m}(t)S=\vec{Tr}(t)$,
$\vec{m}(t)=\vec{Tr}(t)S^{-1}$; and so, \[m_{\lambda_l}(t) =
\frac{\chi_{\lambda_l}(C_1)Tr_{\sigma_1}}{\sum_i
\chi_{\lambda_i}(C_1)^2}+ \cdots +
\frac{\chi_{\lambda_l}(C_k)Tr_{\sigma_k}}{\sum_i
\chi_{\lambda_i}(C_k)^2}.\] However, $\sum_i \chi_{\lambda_i}(C_j)^2
= [S_n:\mathcal{C}(j)] = n!/|\mathcal{C}(j)|$, where
$|\mathcal{C}(j)|$ is the size of the conjugacy class of partition
$j$. Thus, \[m_{\lambda_l,r} = \left[ \frac{1}{n!} \displaystyle
\sum_{j = \text{partition of
n}}\chi_{\lambda_l}(C_j)|\mathcal{C}(j)| Tr_{(j)}
\right]_{(l_1,...,l_k,r)} \qedhere\]
\end{proof}

\section{Representations of Aut$(A(\Gamma))$ acting on
$A(\Gamma)^!$}\label{S:dual}

We will use the same methodology as for $A(\Gamma)$ to determine the
irreducible representations that are realized in $A(\Gamma)^!$. See
Section \ref{S:subalg} for the definition of the dual.  However, we
will see that $Tr_\sigma(A(\Gamma)^{!},t)$ has negative coefficients
and so is not a generating function of a graded dimension, unlike in
the case of gr$A(\Gamma)$.

\subsection{Representations of Aut$(A(\Gamma_{D_n}))$ acting
on $A(\Gamma_{D_n})^!$}

\begin{prop} The set $\{u^*,\,v_{ii+1}^*,\,w_i^*$ $1\leq i \leq n$, *, $u^*v_{ii+1}^*$ $1\leq i\leq n-1$, $v_{ii+1}^*w_i^*$ $1\leq i\leq n$, and
$u^*v_{12}^*w_1^*\}$ is a basis for the graded dual algebra $A(\Gamma_{D_n})^{!}$.
\end{prop}

\begin{proof}The generators of $A(\Gamma_{D_n})^!$ are
$u^*,\,v_{ii+1}^*,\,w_i^*$ $1\leq i \leq n$.  In the associated
graded algebra the relations are $v_{ii+1}(w_i-w_{i+1})$ and
$u(v_{ii+1}-v_{i+1i+2})$.  Thus, the relations in the dual are
$u^{*2}, u^*w_i^*, v_{ii+1}^*u^*$, $w_i^*u^*,w_i^*w_j^*,
v_{ii+1}^*v_{jj+1}^*$, $w_i^*v_{jj+1}^*,
u^*(v_{12}^*+\cdots+v_{n1}^*), v_{ii+1}^*w_j^*$ if $j\neq i, i+1$,
and $v_{ii+1}^*(w_i^*+w_{i+1}^*)$. The elements in the graded dual
follow.
\end{proof}

\smallskip
Now let us determine the trace on the graded pieces by seeing how
each conjugacy class acts on the elements in the dual. As $A(\Gamma_{D_n})^! = A(\Gamma_{D_n})^!_{[0]}\oplus A(\Gamma_{D_n})^!_{[1]}\oplus
A(\Gamma_{D_n})^!_{[2]}\oplus A(\Gamma_{D_n})^!_{[3]}$, there are only three degrees in the dual; so, we can calculate each independently.

\noindent Case 1: n=2m is even

\smallskip

The traces on the graded pieces are:

\begin{center}
\begin{tabular}{l|cccccccc}
&1&r&...&$r^j$&...&$r^m$&s&rs\\
\hline $Tr_{\sigma,1}$&2n+1&1&...&1&...&1&3&3\\
$Tr_{\sigma,2}$&2n-1&-1&...&-1&...&-1&-1&-1\\
$Tr_{\sigma,3}$&1&1&...&1&...&1&-1&-1\\
\end{tabular}
\end{center}

\medskip

\noindent Now that we have the graded traces we can find the
multiplicities of the representations by solving the system of
equations: $\sum_{\phi} m_{\phi,i}*\chi_{\sigma
\phi}(x)=Tr_{\sigma,i}(x)$, $x\in D_n$. They are:

\begin{center}
\begin{tabular}{l|ccccc}
&$\chi_{triv}$&$\chi_{1-1}$&$\chi_{-11}$&$\chi_{-1-1}$&$\chi_k$\\
\hline $m_{\phi,1}$&3&0&1&1&2\\
$m_{\phi,2}$&0&1&1&1&2\\
$m_{\phi,3}$&0&-1&0&0&0\\
\end{tabular}
\end{center}

\bigskip

\noindent Case 2: n=2m+1 is odd

\smallskip

\noindent The traces on the graded pieces are:

\begin{center}
\begin{tabular}{l|ccccccc}
&1&r&...&$r^j$&...&$r^m$&s\\
\hline $Tr_{\sigma,1}$&2n+1&1&...&1&...&1&3\\
$Tr_{\sigma,2}$&2n-1&-1&...&-1&...&-1&-1\\
$Tr_{\sigma,3}$&1&1&...&1&...&1&-1\\
\end{tabular}
\end{center}

\smallskip

\noindent The multiplicities are given below:

\begin{center}
\begin{tabular}{l|ccc}
&$\chi_{triv}$&$\chi_{1-1}$&$\chi_k$\\
\hline $m_{\phi,1}$&3&0&2\\
$m_{\phi,2}$&0&1&2\\
$m_{\phi,3}$&0&-1&0\\
\end{tabular}
\end{center}

\smallskip

\noindent Notice that the graded traces and multiplicities are the
same in both the even and odd cases.

\medskip

\noindent These values give graded trace functions (in both even and
odd cases) of:

$Tr_{(1)}(A(\Gamma_{D_n})^!,t)=1+(2n+1)t+(2n-1)t^2+t^3$

$Tr_{r^i}(A(\Gamma_{D_n})^!,t)=1+t-t^2+t^3$

$Tr_s(A(\Gamma_{D_n})^!,t)= Tr_{rs}(A(\Gamma_{D_n})^!,t)=
1+3t-t^2-t^3$

\subsection{Representations of $S_n$ acting on $Q_n^!$}

[\cite{GGRSW},\S 6] determines a basis for $Q_n^!$ as follows:

\noindent Let $A \subseteq \{1,...,n\}$, $B$ be the sequence
$(b_1,...,b_k)$, and $B'=\{b_1,...,b_k\}$.  Define
$S(A:B) = s(A)s(A\backslash b_1)\cdots s(A\backslash
b_1\backslash...\backslash b_k)$ where $s(A)$ is the image in
$Q_n^!$ of the generator dual to $e(A,1)\in Q_n$. Then
$\mathcal{S}=\{S(A:B) | \text{min}A \notin B \text{ and }b_1>\cdots
>b_k\} \cup \{\emptyset\}$ is a basis for $Q_n^!$. The relations in
the associated graded dual are:

1) $s(A)\sum_{a\in A} s(A\backslash a)=0, |A|\geq2$

2) $s(A)s(A\backslash i)s(A\backslash i\backslash j) =
-s(A)s(A\backslash j)s(A\backslash i\backslash j)$

3) $s(A)s(B)=0$ if $B\nsubseteq A$ or $|B| \neq |A|-1$.

As opposed to the case of $Q_n$, $\sigma$ does not permute the basis
elements of $Q_n^!$.  Thus, it is not enough to count fixed basis
elements to determine the trace.  For each $S(A:B) \in \mathcal{S}$,
we must write $\sigma S(A:B)$ as a linear combination of elements
$\mathcal{S}$. Write this as $\sigma S(A:B) = S(\sigma A:\sigma B) =
\displaystyle \sum_{S(C:D)\in \mathcal{S}}a_{\sigma A\sigma
BCD}S(C:D)$. Then Tr$_\sigma = \displaystyle
\sum_{S(A:B)\in\mathcal{S}} a_{\sigma A\sigma BAB}$. We are going to
get three possible values for a basis element's contribution to the
trace: $-1,0, \text{ or }1$.

If $B$ is $\sigma$-invariant, then let $l_B(\sigma)$ be the number
of pairs $i, \, j$ with $i<j$ and $\sigma b_i <\sigma b_j$.  This is
the length of $\sigma$ restricted to $B$.  If there exists $c\in B$
such that $\sigma(c)=min A$, then define $\sigma':=(c \, min
A)\sigma$.

\begin{prop}\label{P:qndualtr}\[a_{\sigma A\sigma BAB} = \begin{cases}
(-1)^{l_B(\sigma)}& \text{if } \sigma A=A, minA \notin \sigma B', \, \sigma B'=B'\\
(-1)^{l_B(\sigma')+1}& \mbox{if $\sigma A=A, minA\in \sigma B'$ and
for some $b\in B'$, $\sigma(B'\backslash b)=B'\backslash
\text{min}A$}\\0& \mbox{otherwise} \end{cases}
\]  \end{prop}

\begin{proof}
If $\sigma \in S_n$ and $\sigma B'=B'$, then by relation (2) above
we have that \begin{equation}\label{Eq:rel2} S(A:\sigma B) =
(-1)^{l_B(\sigma)}S(A:B).\end{equation}  If min $A\in B'$, then by
relation (1) \begin{equation}\label{Eq:rel1} S(A:B) = \displaystyle
-\sum_{c\in A\backslash B'} S(A:(b_1,...,b_{k-1},c)).\end{equation}

Let us break this down into parts.

\noindent (*) $a_{ABCD} = 0$ if $C\neq A$.  This is true because no
relation changes the first factor, $s(C)$, of $S(C:D)$.

\noindent (**) $a_{ABCD} = 0$ unless $B'=D'$ or $B'=(B'\cap D')\cup
\{min A\}$.  Only relation (1) can change which elements are
removed, and that relation can only change one element.

Recall $\sigma S(A:B) = S(\sigma A: \sigma B) = \displaystyle
\sum_{S(C:D)\in \mathcal{S}} a_{\sigma A\sigma B C D} S(C:D)$.  We
need to know the value of $a_{\sigma A\sigma B AB}$. By (*) and
(**), this is 0 unless $\sigma A = A$ and $\sigma B'=B'$ or $\sigma
B' = (B'\cap \sigma B')\cup \{min A\}$. If $\sigma A = A$ and
$\sigma B'=B'$, then, by Equation \ref{Eq:rel2}, $a_{\sigma A\sigma
B AB} = (-1)^{l_B(\sigma)}$.  If $\sigma A = A$ and $\sigma B' =
(B'\cap \sigma B')\cup \{min A\}$, write $\sigma(b_j)=min A$.  Then
$(b_j\, minA)\sigma B' =\sigma'B'= B'$.  Thus, by Equation
\ref{Eq:rel1}, $\sigma S(A:B) = \displaystyle
-S(A:(\sigma(b_1),...,\hat{b_j},...,\sigma(b_{k}),b_j))+ \text{other
terms}$.  And, again by Equation \ref{Eq:rel2}, $\sigma S(A:B) =
(-1)^{l_B(\sigma')}(-S(A:B))$.  Hence, $a_{\sigma A\sigma B AB} =
(-1)^{l_B(\sigma)+1}$.
\end{proof}

Now that we know what each basis element contributes to the trace,
we want to find $Tr_\sigma(Q_n^!,t)$.

Let us introduce some notation.  For $\sigma \in S_n$, write
$\sigma=\sigma_1\cdots \sigma_m$, a product of disjoint cycles.
Denote the orbits of $\sigma$ by $\{\mathcal{O}_1,
\mathcal{O}_2,...,\mathcal{O}_m\}$ and put an ordering on the orbits
given by $\mathcal{O}_i<\mathcal{O}_j$ if the minimal element of
$\mathcal{O}_i$ is less than that of $\mathcal{O}_j$.  Say
$\mathcal{O}_1< \mathcal{O}_2< \cdots < \mathcal{O}_m$. Let $i_j$ be
the size of $\mathcal{O}_j$ (equal to the length of $\sigma_j$).

\begin{thm}\label{T:qndualtr}\[ Tr_\sigma(Q_n^!,t) =
\frac{1+t\displaystyle \prod_{k=1}^m (2-(-t)^{i_k})}{1+t} \]
\end{thm}

\begin{proof}
In order to prove the formula, we must take the sum over all
$S(A:B)\in \mathcal{S}$ of $a_{\sigma A\sigma BAB}$, each of their
contribution to the trace.  We will do this in cases based on the
value of the basis element's contribution to the trace.

Case 1: $\sigma A=A$, min$A \notin \sigma B', \, \sigma B'=B'$

Consider $B=\mathcal{O}_{r_1} \cup ... \cup \mathcal{O}_{r_l}$,
where $r_1<\cdots <r_l$.  Because $B'\subseteq A$ and min$A\notin
\sigma B'$, $A $ must contain all of
$\mathcal{O}_{r_1},...,\mathcal{O}_{r_l}$ and at least one
$\mathcal{O}_{r_0}$ such that $r_0<r_1$ (so, $r_1\neq 1$).  Thus we
must choose a nonempty subset of
$\{\mathcal{O}_1,...,\mathcal{O}_{r_1-1}\}$ and a subset $ A'$ of
$\{\mathcal{O}_{r_1+1},...,\mathcal{O}_m\}$ such that
$A'\cap\{\mathcal{O}_{r_1+1},...,\mathcal{O}_m\} = \emptyset$. This
gives $2^{m-r_1-(l-1)}(2^{r_1-1}-1) = 2^{m-l}-2^{m-r_1-l+1}$ choices
for A. Because $l_B(\sigma_{r_i})=i_{r_i}-1$, $l_B(\sigma) =
\sum_{1\leq i\leq l} l_B(\sigma_{r_i}) = i_{r_1}+ \cdots
+i_{r_l}-l$.  Also, the degree of $S(A:B)$ is $i_{r_1}+ \cdots
+i_{r_l}+1$.  Thus the contribution toward $Tr_\sigma$ for all such
$S(A:B)$ given $B$ is $d(B) =
(-1)^{i_{r_1}+\cdots+i_{r_l}-l}[2^{m-l}-2^{m-r_1-l+1}]t^{i_{r_1}+\cdots+i_{r_l}+1}.$
Summing over all $B$ in this case, we obtain \[\sum_Bd(B) =
\sum_{\substack{2\leq r_1<\cdots<r_l\leq m\\1\leq l\leq m-1}}
(-1)^{i_{r_1}+\cdots+i_{r_l}-l}[2^{m-l}-2^{m-r_1-l+1}]
t^{i_{r_1}+\cdots+i_{r_l}+1}.\] We will denote this by $c_1$ for ease
of referencing later.

\noindent Case 2: $\sigma A=A$, min$A\in \sigma B'$ and for some
$b\in B'$, $\sigma(B'\backslash b)=B'\backslash \text{min}A$

Fix $B$, say
$B'\subset\mathcal{O}_{r_1}\cup...\cup\mathcal{O}_{r_l}$.  $B'$ must
contain all elements in $\{\mathcal{O}_{r_2},...
,\mathcal{O}_{r_l}\}$ since $B'$ and $\sigma B'$ can only differ by
one element; and, that must occur in $\mathcal{O}_{r_1}$ because
min$A$ must be in $\mathcal{O}_{r_1}$ and cannot be in $B'$. Say
$\sigma_{r_1}=(c_{i_{r_1}-1}\cdots c_1 \, min A)$.  Then $B$ must also
contain consecutive elements $\{c_1,...,c_j\},\, 1\leq j\leq
i_{r_1}-1$, in $\mathcal{O}_{r_1}$.  If this were not the case, $B'$
and $\sigma B'$ would differ by more than one element ($c_j \notin
\sigma B$).

Consider $\sigma' = (c_j \, min A)(min
A\,c_{i_{r_1}-1}\cdots c_1)\sigma_{r_2}\cdots\sigma_{r_1} = \\(c_j \, min
A)(min A\, c_1)\cdots (min
A\,c_{i_{r_1}-1})\sigma_{r_2}\cdots \sigma_{r_1}$.  Then $l_B(\sigma') =
\displaystyle \sum_{k=2}^l(i_{r_k}-1)+j+1$.  Thus, by Proposition
\ref{P:qndualtr}, the trace of $\sigma$ acting on $S(A:B)$ is
$(-1)^{j+i_{r_2}+\cdots+i_{r_l}-(l-1)+2} =
(-1)^{j+i_{r_2}+\cdots+i_{r_l}-(l-1)}$.

Given $B$, $A$ must contain
$\{\mathcal{O}_{r_1},...,\mathcal{O}_{r_l}\}$ and may contain
other orbits greater than $\mathcal{O}_{r_1}$.  Thus, there are $2^{m-r_1-(l-1)}$
choices for $A$.  Now there are $i_{r_1}-1$ subsets
$B\subset\mathcal{O}_{r_1}\cup...\cup\mathcal{O}_{r_l}$.  Putting this all together, given
$\{\mathcal{O}_{r_1},...,\mathcal{O}_{r_l}\}$, $S(A:B)$ contributes
a total of \\$2^{m-r_1-(l-1)} \displaystyle \sum_{j=1}^{i_{r_1}-1}
(-1)^{j+i_{r_2}+\cdots+i_{r_l}-(l-1)}t^{j+i_{r_2}+\cdots+i_{r_l}+1}$
towards the graded trace function.

We will need to sum this over all
$\{\mathcal{O}_{r_1},...,\mathcal{O}_{r_l}\}$ and multiply by $1+t$.
This gives us \begin{eqnarray*}&\displaystyle\sum_{\substack{1\leq
r_1<\cdots<r_l\leq m\\1\leq l\leq m}} 2^{m-r_1-l+1}
(-1)^{i_{r_2}+\cdots+i_{r_l}-l}\,t^{i_{r_2}+\cdots +i_{r_l}+2}\\&+
\displaystyle\sum_{\substack{1\leq r_1<\cdots<r_l\leq m\\1\leq l\leq
m}} 2^{m-r_1-l+1}
(-1)^{i_{r_1}+\cdots+i_{r_l}-l}\,t^{i_{r_1}+\cdots+i_{r_l}+1}\end{eqnarray*}
(notice that the sum over $j$ is telescoping.)  Let us label the
first sum by $c_2$ and the second by $c_3$ for ease of referencing
later.

\noindent Case 3: $B={\emptyset}$.

Because $\sigma A=A$, $a_{\sigma A\sigma BAB}=1$.  Thus we have a
contribution of $1+(2^m-1)t$ towards the graded trace. Multiplying
by $1+t$ gives us $1 + 2^mt + (2^m-1)t^2$.

If we sum over all possibilities for the traces and multiply by
$1+t$, we have that $Tr_\sigma(Q_n^!,t) = 1 + 2^mt + (2^m-1)t^2 +
c_1 +c_1t + c_2 + c_3.$

\noindent Consider the following pieces of the expression.

\noindent First sum $(2^m-1)t^2$ and the $l=1$ terms of $c_2$.

$(2^m-1)t^2+c_2|_{l=1} =(2^m-1)t^2 + \displaystyle \sum_{1\leq
r_1\leq m} 2^{m-r_1}(-1)^{-1}t^2$

\qquad \qquad \qquad $= t^2[(2^m-1) - \displaystyle \sum_{r_1=1}^m
2^{m-r_1}] = t^2[(2^m-1)-(2^{m-1+1}-1)] = 0.$

\noindent Next sum the remaining terms of $c_2$ (with $l>1$) and
$tc_1$ where we do a change of variables setting $r_1$ to $r_2$.

$c_1t|_{r_1\mapsto r_2}$ + $c_2|_{l>1} =$
\begin{align*}&(2^{m-l+1}-2^{m-r_2-l+1+1})
(-1)^{i_{r_2}+\cdots+i_{r_l}-(l-1)}t^{i_{r_2}+\cdots+i_{r_l}+2}\\& +
\sum_{r_1=1}^{r_2-1}2^{m-r_1-l+1} (-1)^{i_{r_2}+\cdots+i_{r_l}-l}
t^{i_{r_2}+\cdots+i_{r_l}+2}\\&= (-1)^{i_{r_2}+\cdots+i_{r_l}-l+1}\,
t^{i_{r_2}+\cdots+i_{r_l}+2} [2^{m-l+1}- 2^{m-r_2-l+2}
-\sum_{r_1=0}^{r_2-2}2^{m-r_1-l}]\\&
=(-1)^{i_{r_2}+\cdots+i_{r_l}-l+1}\, t^{i_{r_2}+\cdots+i_{r_l}+2}
[2^{m-l+1}- 2^{m-r_2-l+2}-[\sum_{r_1=0}^{m-l}2^{r_1} -
\sum_{r_1=0}^{m-r_2-l+1}2^{r_1}]]\\& =
(-1)^{i_{r_2}+\cdots+i_{r_l}-l+1}\, t^{i_{r_2}+\cdots+i_{r_l}+2}
[2^{m-l+1}- 2^{m-r_2-l+2} -(2^{m-l+1}-1) + (2^{m-(r_2-1)-l+1}-1)]\\&
= 0.\end{align*}

\noindent Finally sum $c_1$ and the terms of $c_3$ with $r_1\neq 1$.

$c_1+c_3|_{r_1\neq 1} = \displaystyle \sum_{\substack{1\leq
r_1<\cdots<r_l\leq m\\1\leq l\leq m-1}}(-1)^{i_{r_1}+\cdots+i_{r_l}-l}
t^{i_{r_1}+\cdots+i_{r_l}+1}[2^{m-l} - 2^{m-r_1-l+1} - 2^{m-r_1-l+1}]$

\qquad \qquad  $=\displaystyle \sum_{\substack{1\leq r_1<\cdots<r_l\leq
m\\1\leq l\leq m-1}}(-1)^{i_{r_1}+\cdots+i_{r_l}-l}\,
t^{i_{r_1}+\cdots+i_{r_l}+1}2^{m-l}$.

\noindent The terms of $c_3$ with $r_1=1$ are: $\displaystyle
\sum_{1\leq l\leq m}(-1)^{i_{1}+\cdots+i_{r_l}-l}\,
t^{i_{1}+\cdots+i_{r_l}+1}2^{m-l}.$

\noindent Putting it all together we obtain: \[1+t[2^m +
\sum_{\substack{1\leq r_1<\cdots<r_l\leq m\\1\leq l\leq m}}(-1)^l
2^{m-l} (-t)^{i_{r_1}+\cdots+i_{r_l}}] = 1 +
t\prod_{k=1}^{m}(2-(-t)^{i_k})\]

\noindent Therefore, \[ Tr_\sigma(Q_n^!,t) = \frac{1+t\displaystyle
\prod_{k=1}^m (2-(-t)^{i_k})}{1+t} \] as desired.
\end{proof}

\begin{eg}Here are the graded trace functions for $Q_4^!$:
\[Tr_{(1)}(Q_4^!,t) = \frac{1+t(2+t)^4}{1+t} =
1+15t+17t^2+7t^3+t^4\] \[Tr_{(12)}(Q_4^!,t) =
\frac{1+t(2-t^2)(2+t)^2}{1+t} = 1+7t+t^2-3t^3-t^4\]
\[Tr_{(123)}(Q_4^!,t) = \frac{1+t(2+t^3)(2+t)}{1+t} =
1+3t-t^2+t^3+t^4 \] \[Tr_{(12)(34)}(Q_4^!,t) =
\frac{1+t(2-t^2)^2}{1+t} = 1+3t-3t^2-t^3+t^4\]
\[Tr_{(1234)}(Q_4^!,t) = \frac{1+t(2-t^4)}{1+t} =
1+t-t^2+t^3-t^4\]\end{eg}

\medskip

Now, to get the representations we do the same as in the case of the
algebra. We have that $\vec{m}(t) = (S^T)^{-1}\vec{Tr}(t)$ and
Proposition \ref{P:frob} are still true if you replace
$Tr_\sigma(Q_n,t)$ with $Tr_\sigma(Q_n^!,t)$.

\begin{eg}Irreducible Representations of $S_4$ acting on $Q_4^!$:

There are only four degrees in the dual, so we can give all of the
multiplicities:
\smallskip

\begin{center}
\begin{tabular}{l|ccccc}
&$\chi_{triv}$&$\chi_{sgn}$&$\chi_{3}$&$\chi_{reg}$&$\chi_{sgn\otimes reg}$\\
\hline $m_{\phi,1}$&4&0&1&3&0\\
$m_{\phi,2}$&0&0&1&3&2\\
$m_{\phi,3}$&0&1&0&0&2\\
$m_{\phi,4}$&0&1&0&0&0
\end{tabular}
\end{center}
\end{eg}

\smallskip
Notice that all of the representations are realized in at least one
grading, but not in every.  Also, each representation occurs with a
much smaller multiplicity than in the algebra. %more comments

\section{Koszulity property}\label{S:koszul}

An interesting property of quadratic algebras is Koszulity.  One of
many equivalent definitions of Koszulity is a lattice definition
\cite{F}.

\begin{definition}[Koszul Algebra]\cite{B} Let $A=(V,R)$ be a
quadratic algebra where $V$ is the span of the generators and $R$
the span of the generating relations in $V\otimes V$.  Then $A$ is
Koszul if the collection of subspaces $\{V^{\otimes i}\otimes
R\otimes V^{\otimes n-i-2}, \, 0\leq i\leq n-2\}$ generates a
distributive lattice in $V^{\otimes n}$ for any n.\end{definition}

One property of Koszul algebras is that the Hilbert series of the
algebra and its dual are related by $H(A,t)*H(A^!,-t)=1$. This
property, however, is not equivalent to Koszulity.  One can easily
check that the analogous property holds for the graded trace
functions that we found for $A(\Gamma)$ and its dual $A(\Gamma)^{!}$
in our two algebras. Namely, $Tr_\sigma (A(\Gamma),t)*Tr_\sigma
(A(\Gamma)^!,-t)$=1 where $\sigma$ is an element in the automorphism
group of the algebra.

\newpage
\nocite{GGRW,RSW3,Pr,RW}

\bibliography{paperbib}{}
\bibliographystyle{pgendstyle}

\end{document}